\documentclass[12pt]{article}
\usepackage[utf8]{inputenc}
\usepackage{mathtools}
\usepackage{amsmath}
\usepackage{amssymb}
\usepackage{mathrsfs}
\usepackage{amsthm}
\usepackage{bbold}
\usepackage{esint}
\usepackage{enumerate}
\usepackage{bm}

\usepackage{hyperref}

\usepackage{titlefoot}

\allowdisplaybreaks

\usepackage{tikz}
\usepackage{xcolor}
\usetikzlibrary{calc,intersections,through,backgrounds}
\usetikzlibrary{decorations.pathreplacing,decorations.pathmorphing,decorations.markings,decorations.shapes}
\usetikzlibrary{shapes}
\usetikzlibrary{external}

\newtheorem{Theorem}{Theorem}
\newtheorem{Corollary}[Theorem]{Corollary}
\newtheorem{Lemma}[Theorem]{Lemma}
\newtheorem{Proposition}[Theorem]{Proposition}

\theoremstyle{definition}

\newtheorem{Remark}[Theorem]{Remark}
\numberwithin{Theorem}{section}

\newcommand{\norm}[1]{\left\|{#1}\right\|}
\newcommand{\wick}[1]{:\hspace{-2pt}{#1}\hspace{-2pt}:}
\newcommand{\jap}[1]{\left\langle{#1}\right\rangle}

\newcommand{\les}{\lesssim}
\newcommand{\into}{\hookrightarrow}
\newcommand{\X}{\mathcal X^\alpha}
\renewcommand{\bar}{\overline}
\renewcommand{\epsilon}{\varepsilon}
\newcommand{\T}{\mathbb T}
\newcommand{\Tt}{\mathcal T_t}
\renewcommand{\H}{\mathcal H}
\newcommand{\N}{\mathbb N}
\renewcommand{\P}{\mathbb P}
\newcommand{\E}{\mathbb E}
\newcommand{\Z}{\mathbb{Z}}
\newcommand{\R}{\mathbb R}
\newcommand{\C}{\mathscr C}
\newcommand{\W}{\mathscr W}
\renewcommand{\d}{\mathrm d}
\newcommand{\1}{\mathbb 1}

\renewcommand{\u}{{\mathbf u}}
\renewcommand{\v}{{\mathbf v}}
\newcommand{\w}{{\mathbf w}}
\newcommand{\0}{\mathbf 0}
\newcommand{\dual}[2]{\left\langle#1,#2\right\rangle}
\renewcommand{\vec}[2]{\begin{pmatrix} #1 \\ #2 \end{pmatrix}}
\newcommand{\A}{\mathcal A}

\newcommand{\f}{\mathbf f}
\newcommand{\g}{\mathbf g}

\let \div \relax

\DeclareMathOperator{\div}{div}

\DeclareMathOperator{\Lip}{Lip}
\DeclareMathOperator{\id}{id}

\makeatletter
\def\DeclareSymbol#1#2#3{\expandafter\gdef\csname MH@symb@#1\endcsname{\tikz[baseline=#2, scale=.18]{#3}}}
\def\<#1>{\ensuremath{\mathchoice{\tikzsetnextfilename{macros#1}{\color{black}\csname MH@symb@#1\endcsname}}{\tikzsetnextfilename{macros#1}{\color{black}\csname MH@symb@#1\endcsname}}{\tikzsetnextfilename{macros#1}\scalebox{.7}{\color{black}\csname MH@symb@#1\endcsname}}
{\tikzsetnextfilename{macros#1}\scalebox{.5}{\color{black}\csname MH@symb@#1\endcsname}}}} 
\makeatother %@->12 (not_normal_char)
\DeclareSymbol{1}{0}{\path (-.4,0) -- (.4,0); \draw[line width=0.2mm] (0,0) -- (0,1.2) 
node[circle, fill, draw, solid, inner sep=0pt, minimum size=1.8pt] {};}

\newcommand{\Stick}{\mathbf{\<1>}}
\newcommand{\stick}{\<1>}
\title{Unique ergodicity for a class of stochastic hyperbolic equations with additive space-time white noise}
\author{Leonardo Tolomeo}
\date{}
\begin{document}

\maketitle

\begin{abstract}
In this paper, we consider a certain class of second order nonlinear PDEs with damping and space-time white noise forcing, posed on the $d$-dimensional torus. This class includes the wave equation for $d=1$ and the beam equation for $d\le 3$. We show that the Gibbs measure is the unique invariant measure for this system. Since the flow does not satisfy the Strong Feller property, we introduce a new technique for showing unique ergodicity. This approach may be also useful in situations in which finite-time blowup is possible.
\end{abstract}
\unmarkedfntext{\textbf{Keywords}: stochastic nonlinear wave equation; stochastic nonlinear beam equation; white noise; ergodicity}
\unmarkedfntext{\textbf{2010 Mathematics Subject Classification}: 35L15, 37A25, 60H15}
\section{Introduction}
Consider the equation
\begin{equation*} \label{maineq}
u_{tt} + u_t + u + (-\Delta)^\frac s2 u + u^3 = \sqrt2\xi,
\end{equation*}
posed on the $d$ - dimensional torus $\T^d$, where $\xi$ is the \emph{space-time white noise} on $\R\times\T^d$ (defined in Section 2), and $s>d$.

By expressing this equation in vectorial notation, 
\begin{equation}\label{general}
\partial_t \vec{u}{u_t} = - \begin{pmatrix}
0 & -1 \\
1 + (-\Delta)^\frac s2 & 1
\end{pmatrix}
\vec{u}{u_t} - \vec{0}{u^3} + \vec{0}{\sqrt2\xi},
\end{equation}
from a formal computation, we expect this system to preserve
the \emph{Gibbs} measure
\begin{equation*}
\d \rho(u,u_t) ``=" \exp\Big(-\frac14 \int u^4 - \frac12 \int u^2 + |(-\Delta)^\frac s2 u|^2 \Big) 
\exp\Big(-\frac12 \int u_t^2\Big) ``\d u\d u_t",
\end{equation*}
where $``\d u\d u_t"$ is the non-existent Lebesgue measure on an infinite dimensional vector space (of functions). 
Heuristically, 
we expect invariance for this measure by splitting \eqref{general} into
\begin{enumerate}
\item 
\begin{equation*}
\partial_t \vec{u}{u_t} = - \begin{pmatrix}
0 & -1 \\
1 + (-\Delta)^\frac s2 & 0
\end{pmatrix}
\vec{u}{u_t} - \vec{0}{u^3},
\end{equation*}
which is a Hamiltonian PDE in the variables $u,u_t$, and so it should preserve 
the Gibbs measure 
$$\exp\Big(-H(u,u_t)\Big) ``\d u\d u_t",$$
where $H(u,u_t) = \frac14 \int u^4 + \frac12 \int u^2 + |(-\Delta)^\frac s2 u|^2 + \frac12 \int u_t^2$,
\item
\begin{equation*}
\partial_t \vec{u}{u_t} = - \begin{pmatrix}
0 & 0 \\
0 & 1
\end{pmatrix}
\vec{u}{u_t} - \vec{0}{\sqrt2 \xi},
\end{equation*}
which is the Ornstein - Uhlenbeck process in the variable $u_t$, and so it preserves the spatial white noise
$$\exp\Big(-\frac12 \int u_t^2\Big)``\d u_t".$$
\end{enumerate}
For $s=1$, up to the damping term $\exp\Big(-\frac12 \int u_t^2\Big) \d u_t$,  the measure $\rho$ corresponds to the well known $\Phi^4_d$ model of quantum field theory,
which is known to be definable without resorting to renormalisation just for $d=1$ (this measure will be rigorously defined - in the case $s>d$ - in Section 2).

Our goal is to study the global behaviour of the flow of \eqref{general}, 
by proving invariance of the measure $\rho$ and furthermore showing that $\rho$ is the unique invariant measure for the flow.

Following ideas first appearing in Bourgain's seminal paper \cite{b94} and in the works of McKean-Vasinski \cite{mv94} and McKean \cite{m95III, m95IV}, there have been many 
developments in proving invariance of the Gibbs measure for \emph{deterministic} dispersive PDEs (see for instance 
\cite{b96, b97, bb14, bb14I, bb14II, btt18, cd15, kmv19, ot18, otw18, r16}).

A natural question that arises when an invariant measure is present is uniqueness of the invariant measure and convergence to equilibrium starting from a ``good enough" initial data. This has been extensively studied in
the case of parabolic stochastic PDEs (see for instance \cite{dpez95, dpz96, eh01, fq15, h09, hm17, rzz17, tw17} and references therein) and for stochastic equations coming from fluid dynamics (see \cite{d13, hm11, zz15} and references therein). 
On the other hand, there are not many results in the dispersive setting, and they often rely either on some smoother version of the noise $\xi$, or onto some dissipative properties of the system (see for instance \cite{do05, ikm17, k02, ks00, ks01, kps02, ks02, ks04, kv13, o06}). 
Indeed, as far as the author knows, the ergodicity result of this paper is the first that can deal with a forcing as rough as space-time white noise in a setting without any dissipation. More precisely, we will prove the following: 
\begin{Theorem}\label{mainthm}
Let $s=4,d=3$. Then the measure $\rho$ is invariant for the Markov process associated to the flow $\Phi_t(\cdot,\xi)$ of \eqref{general}, in the sense that for every
function $F$ measurable and bounded, for $\u = (u,u_t)^\textup{T}$,
\begin{equation} \label{defInvariance}
\int \E[F(\Phi_t(\u,\xi))]\d\rho(\u) = \int F(\u) \d\rho(\u) \text{ for every } t>0. 
\end{equation}
Moreover, there exists a Banach space $X^\alpha$ which contains the Sobolev Space $\H^\frac s2 := H^\frac s2 \times L^2$, such that for every $0<\alpha< \frac {s-d}2$, $\rho$ is the only invariant measure concentrated on $X^\alpha$. Furthermore, for every $\u_0 \in X^\alpha$ and for every $F: X^\alpha \to \R$ continuous and bounded,
\begin{equation}\label{meanconvergence}
\lim_{T \to \infty} \frac1T\int_0^T\E[F(\Phi_t(\u_0,\xi))]\d t=\int F(\u) \d\rho(u).
\end{equation}
\end{Theorem}
We will carry out the proof in full details only in the case $s=4, d=3$, however the same proof can be extended to the general nonsingular case $s>d$. We will describe in more details how to obtain the general result in Section 1.4.

The proof of this theorem is heavily influenced by the recent parabolic literature, and in particular by results that use or are inspired by the Bismut-Elworthy-Li formula, especially \cite{tw17} and \cite{hm17}.  
A crucial step in these papers is showing that the flow of the equation in study satisfies the {Strong Feller} property. However, as we will prove in Section 5, the flow of 
\eqref{general} does \emph{not}
satisfy the Strong Feller property, therefore a more refined approach is needed. 

While the argument in this work does not provide any information on the rate of convergence to equilibrium, 
it does not rely on good long time estimates, as opposed to works that rely on the \emph{Asymptotic Strong Feller} property defined in \cite{hm06}. In particular, as far as ergodicity is concerned 
(in the sense that \eqref{meanconvergence} holds $\rho$-a.s.), 
we use just the qualitative result of global existence of the flow, and it may be possible to extend this approach even to situations in which finite-time blowup is possible, similarly to the result in \cite{hm17}. This goes in the direction of dealing with the \emph{singular} case $s=d$. Indeed, in the case $s=d=2$, in a upcoming work by  M.Gubinelli, H. Koch, T. Oh and the author, we prove  global well posedness and invariance of the Gibbs measure for the renormalised damped wave equation
\begin{equation*}
u_{tt} + u_t + u - \Delta u + \wick{u^3} = \sqrt2 \xi.
\end{equation*}
See also \cite{tphd,tr2} for the global existence part of the argument.

However, since the best bound available with the argument grows more than exponentially in time, any approach on unique ergodicity that relies on good long time estimate has little chance to yield any result for this equation.

\subsection{Structure of the argument and organisation of the paper}
In order to make this paper less notation-heavy, we will concentrate on the case $d=3$, $s=4$, which is the Beam equation in 3-d
\begin{equation} \label{bvec}
\partial_t \begin{pmatrix} u \\ u_t \end{pmatrix} = 
-\begin{pmatrix} 0 & -1 \\ 1 + \Delta^2 & 1\end{pmatrix}\begin{pmatrix} u \\ u_t \end{pmatrix} - \begin{pmatrix} 0 \\ u^3  \end{pmatrix} + \begin{pmatrix} 0 \\ \sqrt2\xi \end{pmatrix}. \\
%&\begin{pmatrix} u \\ u_t \end{pmatrix} (0) &=& \begin{pmatrix} u_0 \\ u_1 \end{pmatrix} \in X^\alpha.
\end{equation} 
Local and global well posedness for the non-damped version of this equation have been explored in detail in an upcoming work by R. Mosincat, O.~Pocovnicu, the author, and Y. Wang (see also \cite{tphd}). We will however present an independent treatment that works for general $s>d$. 
While the case $s=2, d=1$, which corresponds to wave equation in 1 dimension, can arguably be considered more interesting, we decide to focus on \eqref{bvec} because it presents all the difficulties of the general case (namely, the definition of the spaces $X^\alpha$, and some subtleties that come from the multidimensional nature of the equation). At the end of this section, we will discuss how to convert the proof for this case into the proof for the general case.

The paper and the proof of Theorem \ref{mainthm} are organised as follows:
\begin{itemize}
\item In the remaining of this section, we will define what we mean by the flow of \eqref{bvec}, and introduce the spaces $X^\alpha$, the stochastic convolution $\stick_t$, and the notation that we will use throughout the paper.
\item In Section 2, we will state and prove the relevant $X^\alpha$ estimates of the stochastic convolution $\stick_t$, 
as well as define rigorously the measure $\rho$ and prove the related $X^\alpha$ estimates for a generic initial data sampled according to $\rho$.
\item In Section 3, we build the flow, by showing local and global well posedness of the equation \eqref{bvec}. Local well posedness is shown by applying a standard Banach fixed point argument, after reformulating the equation
using the so-called Da Prato-Debussche trick. Global well posedness is shown via an energy estimate, making use of an integration by parts trick similar to the one used in \cite{op16}.
\item In Section 4, we show invariance for the measure $\rho$.
\item Section 5 is dedicated to showing unique ergodicity of $\rho$. In particular, we first recover the Strong Feller 
property by changing the underlying topology of the space $X^\alpha$. However, with this new topology, the space ceases to be connected and separable. Therefore, even when we combine this property with irreducibility of the flow, we derive just the partial information that if $\rho_1\perp \rho_2$ are invariant, then there exists a Borel set $V$ s.t.
$$ \rho_1(V+\H^2)=0, \rho_2(V+\H^2)=1.$$
In order to obtain ergodicity of $\rho$, we combine this argument with a completely algebraic one. We consider the projection $\pi: X^\alpha\to X^\alpha/\H^\frac s2$, and we show that if $\rho_1,\rho_2 \ll \rho$, then $\pi_\sharp \rho_1 = \pi_\sharp \rho_2=\pi_\sharp \rho$, which contradicts the existence of such $V$.

Finally, to conclude uniqueness, we show that for every $\u_0 \in X^\alpha$, if $\mu_t$ is the law of 
$\u(t)=(u(t),u_t(t))$, then every weak limit $\nu$ of $\frac1T\int_0^T \mu_t \d t$ will satisfy $\pi_\sharp \nu = \pi_\sharp \rho$, from which we derive $\nu=\rho$.
\end{itemize}
%
%
%
% 
%What we will actually show is that equipping $X^\alpha$ with a much stronger topology, 
%we can recover the Strong Feller property using an argument similar to the one in \cite{hm17}. 
%The price that we pay by using this topology instead of the normal one is that this way our space ceases to be connected and
%separable. In order to conclude, we will combine this argument with an algebraic one that goes through the completely measure theoretical projection $\pi: X^\alpha \to X^\alpha/\H^\frac s2$. Since this
%projection cannot be made into a continuous map in any sensible way, in this way we lose all the
%possible information regarding the rate of convergence of \eqref{convergence}. However, this approach still has 
%benefits over the \emph{Asymptotic Strong Feller} one in e.g.\ \cite{hm06}. Namely, as far ergodicity of the measure $\rho$ is concerned (and not uniqueness of the invariant measure), this argument is completely local
%in time, and might be adaptable to situations in which there are no good long-time estimates available, or even finite time blowup is possible (as in \cite{hm17}).\\ 
%For simplicity, throughout the paper we will just consider the case of \emph{beam} equation* in three spacial
%dimensions, i.e.\ $s=4$, $d=3$,
%
%with the understanding that every proof extends in a straightforward way to the general case.\\
%More precisely, the bounds $\frac12$ on regularity correspond to bounds by $\frac{d-s}2$ and the space $\H^2$
%should be replaced by $\H^{\frac s2}$.
\subsection{Mild formulation}
Before discussing ergodicity issues, we need to define the flow of \eqref{bvec}. Consider the\ \emph{linear}\  damped beam equation with forcing $\mathbf f = \vec{f}{g}$ and initial data 
$\u_0 = \vec{u_0}{u_1}$,
\begin{equation*} \label{linear_beam} 
\partial_t \begin{pmatrix} u \\ u_t \end{pmatrix} = 
-\begin{pmatrix} 0 & -1 \\ 1 + \Delta^2 & 1\end{pmatrix}\begin{pmatrix} u \\ u_t \end{pmatrix} + \vec{f}{g}. 
\end{equation*}
By variation of constants, the solution to this equation is given by 
\begin{equation}\label{voc}
\u = S(t)\u_0 + \int_0^t S(t-t') \mathbf f(t') \d t',
\end{equation}
where $S(t)$ is the operator formally defined as 
\begin{equation*}
e^{-\frac t2}
\resizebox{0.9\hsize}{!}{$\displaystyle
\begin{pmatrix}\cos\Big(t\sqrt{\frac34+\Delta^2}\Big) + \frac12\frac{\sin\Big(t\sqrt{\frac34+\Delta^2}\Big)}{\sqrt{\frac34+\Delta^2}}
&\frac{\sin\Big(t\sqrt{\frac34+\Delta^2}\Big)}{\sqrt{\frac34+\Delta^2}} \\
-\Big(\sqrt{\frac34+\Delta^2}-\frac1{4\sqrt{\frac34+\Delta^2}}\Big)\sin\Big(t\sqrt{\frac34+\Delta^2}\Big) & 
\cos\Big(t\sqrt{\frac34+\Delta^2}\Big) -\frac12\frac{\sin\Big(t\sqrt{\frac34+\Delta^2}\Big)}{\sqrt{\frac34+\Delta^2}}
\end{pmatrix}$},
\end{equation*}
or equivalently, is the operator that corresponds to the symbol 
\begin{equation*}
e^{-\frac t2}
\resizebox{0.9\hsize}{!}{$\displaystyle
\begin{pmatrix}\cos\Big(t\sqrt{\frac34+|n|^4}\Big) + \frac12\frac{\sin\Big(t\sqrt{\frac34+|n|^4}\Big)}{\sqrt{\frac34+|n|^4}}
&\frac{\sin\Big(t\sqrt{\frac34+|n|^4}\Big)}{\sqrt{\frac34+|n|^4}} \\
-\Big(\sqrt{\frac34+|n|^4}-\frac1{4\sqrt{\frac34+|n|^4}}\Big)\sin\Big(t\sqrt{\frac34+|n|^4}\Big) & 
\cos\Big(t\sqrt{\frac34+|n|^4}\Big) -\frac12\frac{\sin\Big(t\sqrt{\frac34+|n|^4}\Big)}{\sqrt{\frac34+|n|^4}}
\end{pmatrix}$}
\end{equation*}
in Fourier series. We notice that this operator maps distributions to distributions, and for every $\alpha \in \R$,
it maps the Sobolev space $\H^\alpha:= H^\alpha \times H^{\alpha-2}$ into itself, with
the estimate $\norm{S(t)\u}_{\H^\alpha} \lesssim e^{-\frac t2}\norm{\u}_{\H^\alpha}$. 

By the formula 
\eqref{voc}, since we formally have $\mathbf f = - \vec{0}{u^3} + \vec{0}{\xi}$, we expect the solution of 
\eqref{bvec} to satisfy the \emph{Duhamel} formulation
\begin{equation} \label{Duhamel}
 \u = S(t)\u_0 + \int_0^t S(t-t')\vec{0}{\xi(t')}\d t' - \int_0^t S(t-t') \vec{0}{u^3(t')} \d t'.
 \end{equation}
From the previous discussion about $S(t)$, we have that 
\begin{equation*}
\stick_t(\xi) :=  \int_0^t S(t-t')\vec{0}{\sqrt2\xi(t')}\d t' 
\end{equation*}
is a well defined space-time distribution. In the following, when it is not ambiguous, we may omit the argument $\xi$ (i.e. $\stick_t:=\stick_t(\xi)$). We will explore more quantitative estimates about $\stick_t$ in Section 2.

Moreover, it is helpful to consider \eqref{Duhamel} as an equation for the term 
$$\v(t) := \u(t) - S(t)\u_0 - \stick_t(\xi) = - \int_0^t S(t-t') \vec{0}{u^3(t')} \d t'. $$
This is the so called Da Prato - Debussche trick (\cite{dpd02,dpd03}). With a slight abuse of notation, the equation for $\v$ becomes
\begin{equation}
\begin{aligned} \label{veqn}
\v(t) &= - \int_0^t S(t-t') \vec{0}{(S(t')\u_0 + \stick_{t'}(\xi) + v(t'))^3}\d t', \\
\end{aligned}
\end{equation}
where $(S(t')\u_0 + \stick_{t'}(\xi) + v(t'))$ is actually the first component of $(S(t')\u_0 + \stick_{t'}(\xi) + \v(t'))$. 

Following this discussion, we \emph{define} a solution for \eqref{bvec} with initial data $\u_0$ to be
$S(t)\u_0 + \stick_t(\xi) + \v(t)$, where $\v(t)$ solves \eqref{veqn}. 

In order to define a flow, we need a space $X$ such that for every $\u_0 \in X$, we can find a solution for \eqref{veqn}, and $S(t)\u_0 + \stick_t(\xi) + \v(t) \in X$ as well. 
Due to the dispersive nature of the equation, this choice is not as straightforward as in the parabolic case,
where H\"older spaces satisfy all of the required properties. On the other hand, keeping track only of the $\H^\alpha$ regularity of the initial data would hide important information about the solution, namely the gain in regularity of the term $\v$. A good space for this equation turns out to be
\begin{equation*}
\begin{gathered}
\bar X^\alpha := \big\{ \u \in S'(\T^d) \times S'(\T^d) | S(t)\u \in C([0,+\infty); \C^\alpha), \norm{S(t)\u}_{\C^\alpha}\lesssim e^{-\frac t 8} \big\},\\
\norm{u}_{X^\alpha} := \sup_{t\ge0} e^{\frac t8} \norm{S(t)\u}_{\C^\alpha},
\end{gathered}
\end{equation*}
for $0 < \alpha < \frac12$.
Here $\C^\alpha := C^\alpha \times C^{\alpha-2}$. As it is common in these situations, the particular definition
of the H\"older spaces $C^\beta$ for $\beta \not\in (0,1)$ (where they all coincide) does not play any role. In this paper,
we choose to define $\norm{u}_{C^\beta}:= \norm{(1-\Delta)^{\frac\beta2}u}_{L^\infty}$.

As it is defined, the space $\bar X^\alpha$ might not be separable, which is a helpful hypothesis for some measure theoretical considerations in the following. In order to solve this issue, we will denote by $X^\alpha$ the closure of trigonometric 
polynomials in $\bar X^\alpha$. Since we have, for $\alpha'>\alpha$, 
$$\norm{\u-P_{\le N}\u}_{X^{\alpha}} \les N^{-\frac{\alpha'-\alpha}2}\norm{\u}_{X^{\alpha'}},$$
we have that for every $\alpha' > \alpha$, if $\norm{\u}_{X^{\alpha'}} < +\infty$, then $\u \in X^\alpha$. 
\begin{Lemma}
$\bar X^\alpha$ is a Banach space.
\end{Lemma}
\begin{proof}
$\norm{\cdot}_{X^\alpha}$ is clearly a norm, so we just need to show completeness. Let $\u_n$ be a Cauchy sequence in $\bar X^\alpha$. By definition, for every $t$, $S(t)\u_n$ is a Cauchy sequence in $\C^\alpha$, so there exists a limit $S(t)\u_n \to \u(t)$ in $\C^\alpha$. Moreover,
$S(t)$ is a bounded operator in $\H^\alpha$, so one has that 
$$\u(t) = \C^\alpha - \lim S(t)\u_n = \H^\alpha - \lim S(t)\u_n = S(t)(\H^\alpha - \lim \u_n )= S(t)\u(0).$$ 
Lastly,
\begin{align*}
\lim_n \norm{\u_n - \u}_{X^\alpha} &= \lim_n \sup_t e^\frac t8\norm{S(t)\u_n - S(t)\u(0)}_{\C^\alpha} \\
&= \lim_n \sup_t \lim_m  e^\frac t8 \norm{S(t)\u_n - S(t)\u_m}_{\C^\alpha} \\
&\le \lim_n \lim_m \sup_t e^\frac t8 \norm{S(t)\u_n - S(t)\u_m}_{\C^\alpha} \\
& = \lim_n \lim_m \norm{\u_n - \u_m}_{X^\alpha} \\
& = 0.
\end{align*}
\end{proof}

\noindent
Since the operator $S(t)$ is \emph{not} bounded on $\C^\alpha$, the space $X^\alpha$ might appear
mysterious. However, in the next sections, we will see that the term $\stick_t(\xi)$ belongs to $X^\alpha$, as well as almost every initial data according to $\rho$, i.e.\ $\rho(X^\alpha) =~1$. Moreover, we have the following embedding for smooth functions:
\begin{Lemma}\label{H2intoX}
For every $0<\alpha<\frac12$, we have $\H^2 \subset X^\alpha$. Moreover, the identity 
$\id: \H^2 \into X^\alpha$ is a compact operator.
\end{Lemma}
\begin{proof}
Let $\u \in \H^2$. By Sobolev embeddings, 
$$\norm{S(t)\u}_{\C^\alpha} \lesssim \norm{S(t)\u}_{\H^2} \lesssim e^{-\frac t2} \norm{\u}_{\H^2}$$
and given $s \ge 0$, we have 
$$\lim_{t\to s} \norm{S(t)\u-S(s)\u}_{\C^\alpha} \lesssim \limsup_{t\to s} \norm{S(t)\u-S(s)\u}_{\H^2}=0, $$
hence $\u \in X^\alpha$. 

Now let $\u_n$ be a bounded sequence in $\H^2$. By compactness of Sobolev embeddings, up to 
subsequences, $\u_n \to \u$ in $\C^\alpha$ and $\u_n \rightharpoonup \u$ weakly in $\H^s$ for every $s\le 2$. Therefore,
$S(t)\u_n \rightharpoonup S(t)\u$ weakly in $\H^s$ for every $t \ge 0$.

By a diagonal argument, up to subsequences, we have that $S(t)\u_n$ is a converging sequence in $\C^\alpha$ for every $t \in \mathbb Q^+$, so by coherence of the limits, $S(t)\u_n \to S(t)\u$ in $\C^\alpha$ 
for every $t \in \mathbb Q^+$. By the property 
$$\partial_t S(t) = -\begin{pmatrix} 0 & -1 \\ 1 + \Delta^2 & 1\end{pmatrix}S(t),$$ 
we have that $\norm{{S(t)\u-S(s)}\u}_{\H^s}\lesssim |t-s|^\epsilon \norm{\u}_{\H^{s+4\epsilon}}$. Therefore,
by taking $\epsilon$ such that $\alpha + 4\epsilon + \frac32 < 2$, by the Sobolev embedding 
$\H^{2-4\epsilon} \into \C^\alpha$, we have that
$S(t)\u_n \to S(t)\u$ in $\C^\alpha$ for every $t\ge 0 $ and uniformly on compact sets. Finally, for every $T$ we have
\begin{align*}
e^{\frac t8} \norm{S(t)\u_n - S(t)\u}_{\C^\alpha} \lesssim 
e^{\frac T8} \sup_{s\in [0,T]} \norm{S(s)\u_n - S(s)\u}_{\C^\alpha} + e^{-\frac38 T}\sup_n \norm{\u_n}_{\H^2}.
\end{align*}
For $T\gg 1$ big  enough and $n \gg 1$ (depending on $T$), we can make the right hand side arbitrarily small. Therefore, we get $\norm{\u_n-\u}_{X^\alpha} \to 0$ as $n \to \infty$, so $\id$ is compact.
\end{proof}
However, the space $X^\alpha$ is strictly bigger than $\H^2$, and it contains functions at regularity exactly $\alpha$. Indeed, we have
\begin{Lemma} \label{XalphaNotH2}
For every $\alpha_1 > \alpha > 0$, there exists $\u_0 \in X^\alpha$ such that $\u_0 \not \in \H^{\alpha_1}$.
\end{Lemma}
\begin{proof}
Suppose by contradiction that $X^\alpha \subseteq \H^{\alpha_1}$. By the closed graph theorem, this implies that 
\begin{equation}\label{XalphainHalpha1}
\norm{\u}_{\H^{\alpha_1}} \les \norm{\u}_{X^\alpha}.
\end{equation}
For $n \in \Z^3$, consider $\u_n := \vec{e^{in \cdot x}}{0}$. By definition of $S(t)$, 
$$S(t) \u_n = e^{-\frac t 2} \vec{\left(\cos\Big(t\sqrt{\frac34+|n|^4}\Big) + \frac12\frac{\sin\Big(t\sqrt{\frac34+|n|^4}\Big)}{\sqrt{\frac34+|n|^4}}\right)e^{in\cdot x}}{\frac{\sin\Big(t\sqrt{\frac34+|n|^4}\Big)}{\sqrt{\frac34+|n|^4}}e^{i n\cdot x}}.$$
It is easy to check that $\norm{S(t)\u_n}_{\C^\alpha} \sim e^{-\frac t2} \jap{n}^\alpha$, so $\norm{S(t)\u_n}_{X^\alpha} \sim \jap{n}^\alpha$. 
On the other hand, $\norm{\u_n}_{\H^{\alpha_1}} \sim \jap{n}^{\alpha_1}$. By \eqref{XalphainHalpha1}, this implies $\jap{n}^{\alpha_1} \les \jap{n}^{\alpha}$, which is a contradiction.
\end{proof}

\subsection{Truncated system}
In order to prove invariance of the measure $\mu$, it will be helpful to introduce a truncated system. While
many truncations are possible, for this particular class of systems it is helpful to introduce the sharp 
Fourier truncation $P_{\le N}$, $N \in \N \cup \{0\}$ given by 
\begin{equation*}
P_{\le N} \u (x) := \frac1{(2\pi)^3} \sum_{\max_j{|n_j|} \le N} \widehat{\u}(n) e^{in\cdot x}, 
\end{equation*}
 i.e.\ the sharp restriction on the cube $[-N,N]^3$ in Fourier variable. Similarly, we define 
 $P_{>N}:= 1- P_{\le N}$. While this is a somewhat odd choice for the truncation, it has the advantages that
 $P_{>N}$ and $P_{\le N}$ have orthogonal ranges, and 
 $\norm{P_{\le N}u}_{L^p} \lesssim_p \norm{u}_{L^p}$ uniformly in $N$ for every $1<p<+\infty$ (since it corresponds to the composition of the Hilbert transform in every variable).
 
It is convenient for notation to allow also $N=-1$, in which case $P_{\le N} = 0$ and $P_{>N} = \id$. Therefore,
we define the \emph{truncated} system to be 
\begin{equation} \label{bvecN}
\left\{\begin{aligned} 
&\partial_t \begin{pmatrix} u \\ u_t \end{pmatrix} &= &
-\begin{pmatrix} 0 & -1 \\ 1 + \Delta^2 & 1\end{pmatrix}\begin{pmatrix} u \\ u_t \end{pmatrix} - P_{\le N}\begin{pmatrix} 0 \\ (P_{\le N}u)^3  \end{pmatrix} + \begin{pmatrix} 0 \\ \sqrt2\xi \end{pmatrix}, \\
&\begin{pmatrix} u \\ u_t \end{pmatrix} (0) &=& \begin{pmatrix} u_0 \\ u_1 \end{pmatrix} \in X^\alpha.
\end{aligned}\right.
\end{equation}
In a similar fashion to \eqref{bvec}, we will write solutions to this system as $S(t) \u_0 + \stick_t(\xi)  + \v_N$, where
$\v$ solves the equation
\begin{equation}
\begin{aligned} \label{veqnN}
\v_N(t) &= - \int_0^t S(t-t') P_{\le N} \vec{0}{P_{\le N}(S(t')\u_0 + \stick_{t'}(\xi) + v_N(t'))^3}. \\
\end{aligned}
\end{equation}
\subsection{Notation and conversion to the general case}
In the following, $\H^\alpha$ will denote the Sobolev space $H^\alpha \times H^{\alpha-2}$, with norm given by
$$\norm{\u}_{\H^\alpha}^2 :=  \norm{(1-\Delta)^\frac\alpha2 u}_{L^2}^2 + \norm{(1-\Delta)^{\frac\alpha2-1} u_t}_{L^2}^2.$$
Similarly, $\W^{\alpha,p}$ will denote the Sobolev space $W^{\alpha,p} \times W^{\alpha-2,p}$ with norm given by
$$\norm{\u}_{\W^\alpha}^p :=  \norm{(1-\Delta)^\frac\alpha2 u}_{L^p}^p + \norm{(1-\Delta)^{\frac\alpha2-1} u_t}_{L^p}^p $$
and as already discussed, $\C^\alpha := C^\alpha \times C^{\alpha-2}$, with norm given by 
$$\norm{\u}_{\C^\alpha}:= \max\left(\norm{(1-\Delta)^{\frac\alpha2}u}_{L^\infty}, \norm{(1-\Delta)^{\frac\alpha2-1}u_t}_{L^\infty}\right). $$
In order to convert the argument presented in this paper into the one for the general case, we make the following modifications:
$$\H^\alpha :=  H^\alpha \times H^{\alpha-\frac s2},\hspace{6pt} \W^{\alpha,p} = W^{\alpha,p} \times W^{\alpha-\frac s2,p},\hspace{6pt}
\C^\alpha := C^\alpha \times C^{\alpha-\frac s2},$$
with the analogous modifications of the norms. Moreover, $S(t)$ would denote the linear propagator for \eqref{general}, and 
\begin{equation*}
\begin{gathered}
\bar X^\alpha := \big\{ \u | S(t)\u \in C([0,+\infty); \C^\alpha), \norm{S(t)\u}_{\C^\alpha}\lesssim e^{-\frac t 8} \big\},\\
\norm{u}_{X^\alpha} := \sup_{t>0} e^{\frac t8} \norm{S(t)\u}_{\C^\alpha}
\end{gathered}
\end{equation*}
is defined for $0 < \alpha < \frac{s-d}2$.
Moreover, in the following discussion, the space $\H^2$ has to be substituted by the space $\H^\frac s2$, and 
any threshold of the regularity in the form $\alpha < \frac12$ has to be substituted by $\alpha < \frac{s-d}2$.
\section{Stochastic objects} \label{stochastic_section}

This section is dedicated to building the stochastic objects that we will need throughout the paper and 
to proving the relevant estimates about them, in the case $s=4, d=3$. More precisely, in the first subsection we prove that 
$\stick_t \in C([0,+\infty);\C^\alpha)$ and $\stick_t \in X^\alpha$ almost surely. In the second subsection, we
build the Gibbs measure(s) and we prove that they are actually concentrated on $X^\alpha$. 

\subsection{Stochastic convolution}
We will use that the space-time white noise is a distribution-valued random variable such that, for every 
$\phi,\psi \in C^\infty_c(\R\times\T^d)$, 
\begin{equation*}
\E[\dual{\phi}{\xi}\dual{\psi}{\xi}] = \dual{\phi}{\psi}_{L^2(\R\times\T^d)}.
\end{equation*}
\begin{Proposition}\label{stickreg}
For every $\alpha < \frac12$, 
\begin{equation*}
\E\norm{\stick_t }_{\C^\alpha}^2 < +\infty.
\end{equation*}
Moreover, $\stick_t \in C([0,+\infty);\C^\alpha)$ almost surely.
\end{Proposition}
\begin{proof}
For a test function $\f = \vec{f}{f_t}$, define
\begin{equation}\label{eq:covariance}
\gamma(t,s)[\f] := \E \dual{\Stick_t(\xi)}{\f}\dual{\Stick_s(\xi)}{\f},
\end{equation}
where $\dual{\f}{\g} = \int f\overline g + \int f_t\overline{g_t}$. We have that
\begin{align*}
\dual{\Stick_t}{\f} & = \int_0^t \dual{ S(t-t')\vec{0}{\sqrt 2 \xi(t')}}{\f} \\
&=\sqrt 2 \dual{\xi}{\pi_2 S(t-t')^\ast \f}_{L^2_{t,x}},
\end{align*}
where $\pi_2$ is the projection on the second component. Therefore, by definition of $\xi$,
\begin{equation}\label{gamman}
\gamma(t,s)[\f] = 2\int_0^{t \wedge s} \dual{\Re \big(\pi_2 S(t-t')^\ast \f\big)}{\Re \big(\pi_2 S(s-t')^\ast \f\big)}.
\end{equation}
Hence, by boundedness of $S(t)$, we have that $\gamma(t,s)[\f] \lesssim \norm{f}_{H^{-2}}^2 + \norm{f_t}_{L^2}^2$. Moreover, 
since 
$$\partial_t S(t) = -\begin{pmatrix} 0 & -1 \\ 1 + \Delta^2 & 1\end{pmatrix}S(t),$$ 
we have $\Lip(\gamma(\cdot,\cdot)[\f])\lesssim \norm{f}_{L^2}^2 + \norm{f_t}_{H^2}^2$, so by interpolation, for every $0 \le \theta \le 1$, we have 
$$\big|\gamma(t,s)-\gamma(t',s')\big|[\f] \les (\norm{f}_{H^{-2(1+\theta)}}^2 + \norm{f_t}_{H^{2\theta}}^2) (|t-t'| + |s-s'|)^\theta.$$
Therefore, choosing $\theta = \frac{1-2\alpha-\epsilon}{4}$, we have
\begin{align}
&\E \norm{\Stick_{t+h}(\xi)-\Stick_{t}(\xi)}_{\H^\alpha}^2\notag \\
\lesssim &\sum_{n\in\Z^3} \jap{n}^{2\alpha} \Big(\gamma(t+h,t+h) - 2\gamma(t+h,t) + \gamma(t,t)\Big)\left[\vec{e^{in\cdot x}}{0}\right]\notag  \\
&+ \sum_{n\in\Z^3} \jap{n}^{2\alpha} \Big(\gamma(t+h,t+h) - 2\gamma(t+h,t) + \gamma(t,t)\Big)\left[\vec{0}{\jap{n}^{-2}e^{in\cdot x}}\right]
\notag \\
\lesssim &\sum_{n\in\Z^3} \jap{n}^{2\alpha} \jap{n}^{-3-2\alpha-\epsilon} |h|^\frac{1-2\alpha-\epsilon}4 \lesssim |h|^\frac{1-2\alpha-\epsilon}4. \label{stickhalpha}
\end{align}
By translation invariance of the operator $S(t)$, we have that  
$$\E |\jap\nabla^{-\alpha}(\Stick_{t+h}-\Stick_{t})(x)|^2 = \E |\jap\nabla^{-\alpha}(\Stick_{t+h}-\Stick_{t})(y)|^2$$
for every $x,y \in \T$, so 
\begin{align*}
\E |\jap\nabla^{\alpha}(\Stick_{t+h}-\Stick_{t})(x)|^2 &\sim \int_\T \E |\jap\nabla^{\alpha}(\Stick_{t+h}-\Stick_{t})(x)|^2 \d x \\
&\le \E \norm{\Stick_{t+h}-\Stick_{t}}_{\H^\alpha}^2 \\
&\lesssim |h|^\frac{1-2\alpha-\epsilon}4. 
\end{align*}
and similarly 
\begin{equation*}
\E |\jap\nabla^{-2+\alpha}\partial_t(\Stick_{t+h}-\Stick_{t})(x)|^2 \lesssim |h|^\frac{1-2\alpha-\epsilon}4.
\end{equation*}
By hypercontractivity, (or since $\jap\nabla^{-\alpha}(\Stick_{t+h}-\Stick_{t})(x)$ is Gaussian), 
$$\E |\jap\nabla^{\alpha}(\Stick_{t+h}-\Stick_{t})(x)|^p \lesssim_p  \Big(\E |\jap\nabla^{\alpha}(\Stick_{t+h}-\Stick_{t})(x)|^2\Big)^\frac p2 \lesssim |h|^\frac{p(1-2\alpha-\epsilon)}4$$
and 
$$\E |\jap\nabla^{-2+\alpha}\partial_t(\Stick_{t+h}-\Stick_{t})(x)|^p \lesssim_p |h|^\frac{p(1-2\alpha-\epsilon)}4$$
so for $p>q$,
\begin{equation} \label{stickwalpha}
\begin{aligned}
& \E \norm{\Stick_{t+h}-\Stick_{t}}_{\W^{\alpha,q}}^p \\
&= \E  \Big(\int |\jap\nabla^{\alpha}(\Stick_{t+h}-\Stick_{t})(x)|^q \d x + \int |\jap\nabla^{-2+\alpha}\partial_t(\Stick_{t+h}-\Stick_{t})(x)|^q \d x\Big)^\frac pq \\
 &\lesssim \E \int |\jap\nabla^{\alpha}(\Stick_{t+h}-\Stick_{t})(x)|^p +  \E \int |\jap\nabla^{-2+\alpha}\partial_t(\Stick_{t+h}-\Stick_{t})(x)|^p\d x \\
 & \lesssim_p |h|^\frac{p(1-2\alpha-\epsilon)}4.
\end{aligned}
\end{equation}
Therefore, by Kolmogorov Continuity Theorem, if $\alpha < \frac12$, by taking $p$ big enough in 
such a way that $\frac{p(1-2\alpha-\epsilon)}4>1$, 
we have $\Stick_t \in C_t^{\frac{(1-2\alpha-\epsilon)}4 - \frac1p}\W^{\alpha,q}.$ For every $\beta < \alpha$, 
by Sobolev embeddings we can find $q<+\infty$ s.t. $\W^{\alpha,q} \subset \C^{\beta}$.
From this we get that $\Stick_t \in C_t\C^{\beta}$.
\end{proof}
\begin{Proposition} \label{stickX}
For every $t>0$, $\Stick_t \in X^\alpha$ a.s. More precisely,
\begin{equation*}
\sup_{s>0} \norm{e^{\frac s8} S(s)\Stick_t}_{\C^\alpha} < +\infty \text{ a.s.} \label{eq:stickX}
\end{equation*}
for every $\alpha < \frac12$.
\end{Proposition}
\begin{proof}
Let $\f=\vec{f}{f_t}$ be a test function, and let $\tilde\gamma$ to be such that
$$\E\dual{S(r)\Stick_t}{\f}\dual{S(s)\Stick_t}{\f} = e^{-\frac{r+s}2}\tilde\gamma(t,s)[\f].$$
As for \eqref{eq:covariance}, we have the analogous of \eqref{gamman}
\begin{equation*}
\tilde \gamma(t,s)[\f] = 2\int_0^{t} \dual{\Re \pi_2 e^{\frac r2}S(r+t-t')^\ast \f}{\Re \pi_2 e^{\frac s2} S(s+t-t')^\ast \f}.
\end{equation*}
Therefore, exactly as for \eqref{eq:covariance}, we have
$\tilde \gamma(t,s)[\f] \lesssim \norm{f}_{H^{-2}}^2 + \norm{f_t}_{L^2}^2$ and $\Lip(\tilde \gamma(\cdot,\cdot)[\f])\lesssim \norm{f}_{L^2}^2 + \norm{f_t}_{H^2}^2$.
Therefore, proceeding as in \eqref{stickhalpha},
\begin{equation*} 
\E \norm{S(s+h)\Stick_{t}-S(s)\Stick_{t}}_{\H^\alpha}^2 \lesssim e^{-s} |h|^\frac{1-2\alpha-\epsilon}4,
\end{equation*}
and arguing as in \eqref{stickwalpha}, for every $p>q$,
\begin{equation*}
\E \norm{S(s+h)\Stick_{t}-S(s)\Stick_{t}}_{\W^{\alpha,q}}^p \lesssim_p e^{-\frac p2 s} |h|^\frac{p(1-2\alpha-\epsilon)}4.
\end{equation*}
Therefore, by Kolmogorov Continuity Theorem, $S(\cdot)\Stick_t \in C_s \W^{\alpha,q}$ a.s. and 
\begin{equation*}
\E \norm{S(\cdot)\Stick_t}_{C_s([N,N+1]; \W^{\alpha,q})}^p \lesssim_p e^{-\frac p2 N}.
\end{equation*}
Therefore, $\P(\norm{S(\cdot)\Stick_t}_{C_t([N,N+1]; \W^{\alpha,q})} > e^{-\frac N4}) \lesssim_p e^{-\frac p4 N}$,
which is summable in $N$, so by Borel-Cantelli  $\norm{S(\cdot)\Stick_t}_{C_s([N,N+1]; \W^{\alpha,q})} \le e^{-\frac N4}$
definitely. Taking $\beta <~\alpha$ and $q$ big enough, by Sobolev embeddings we have that \linebreak
$\norm{S(\cdot)\Stick_t}_{C_s([N,N+1]; \C^\beta)} \lesssim e^{-\frac N4}$ definitely. Therefore,
$$ \limsup_{s \to \infty} \norm{e^{\frac s8} S(s)\Stick_t}_{\C^\beta} = 0 \text{ a.s.},$$
which in particular implies \eqref{stickX}.
\end{proof}
\subsection{Invariant measure}

Consider the distribution-valued random variable 
\begin{equation} \label{u(omega)}
\u = \vec{\Re\Big(\sum_{n \in \Z^3} \frac{g_n}{\sqrt{1+|n|^4}} e^{inx}\Big)}{\Re\Big(\sum_{n \in \Z^3} h_n e^{inx}\Big)}
\end{equation}
where $g_n,h_n$ are independent complex-valued standard gaussians (i.e. the real and imaginary parts are independent real valued standard gaussians). If $\mathbf f = \vec{f}{f_t}$ is a test function, then 
\begin{equation} \label{eq:wiener}
\begin{aligned}
\E[\dual{\u}{\mathbf f}^2] &= \frac12 \sum_n \frac{\E|g_n|^2}{1+|n|^4} |\hat f(n)|^2 + \frac12 \sum_n \E|h_n|^2 |\hat f_t(n)|^2  \\
	&= \sum_n \frac{|\hat f(n)|^2}{1+|n|^4} + \sum_n |\hat f_t(n)|^2 \\
	&= \norm{(1+\Delta^2)^{-\frac12}f}_{L^2}^2 + \norm{f_t}_{L^2}^2.
\end{aligned}
\end{equation}
Therefore, if $\mu$ is the law of $\u$, we have that formally
\begin{equation}\label{mu}
\begin{aligned}
\d\mu(\u) &= \exp\Big(-\frac12 \norm{(1+\Delta^2)^\frac12u}_{L^2}^2\Big)\d u\, \times \exp\Big(-\frac12\norm{u_t}_{L^2}^2\Big)\d u_t\\
	&= \exp\Big(-\frac12 \int u^2 - \frac12 \int (\Delta u)^2\Big)\d u \, \times \exp\Big(-\frac12\int u_t^2\Big)\d u_t.
\end{aligned}
\end{equation}
\begin{Proposition}
For every $\alpha < \frac12$, $\u \in \C^\alpha$ a.s. \label{muholder}
\end{Proposition}
\begin{proof}
By Sobolev embeddings, it is enough to show that $u \in \W^{\alpha,p}$ a.s. for every 
$p > 0$. We have that 
\begin{equation*}
\E|\jap{\nabla}^\alpha u(x)|^2 = \frac12 \E \Big|\sum_{n \in \Z^3} \jap{n}^\alpha \frac{g_n}{\sqrt{1+|n|^4}} e^{inx}\Big|^2 
		=\frac12 \sum_{n\in \Z^3} \frac {\jap{n}^{2\alpha}}{1+|n|^4} \lesssim_\alpha 1,
\end{equation*}
and similarly 
\begin{equation*}
\E|\jap{\nabla}^{-2+\alpha} u_t(x)|^2 = \frac12 \E \Big|\sum_{n\in\Z^3} \jap{n}^{-2+\alpha} h_n e^{inx}\Big|^2 = 
\frac12\sum_{n\in\Z^3} \jap{n}^{-4+2\alpha} \lesssim_\alpha 1.
\end{equation*}
Therefore, by hypercontractivity, for $q>p$,
\begin{equation*}
\begin{aligned}
\E\norm{u}_{\W^{\alpha,p}}^q &= \E\Big(\int |\jap{\nabla}^\alpha u(x)|^p \d x + \int|\jap{\nabla}^{-2+\alpha} u_t(x)|^p \d x \Big)^\frac qp \\
	&\le \E\Big[\int |\jap{\nabla}^\alpha u(x)|^q \d x + \int|\jap{\nabla}^{-2+\alpha} u_t(x)|^q \d x \Big] \\
	& \lesssim \int (\E|\jap{\nabla}^\alpha u(x)|^2)^\frac q2 \d x + \int (\E|\jap{\nabla}^{-2+\alpha} u_t(x)|^2)^\frac q2 \d x \\
	& \lesssim_\alpha 1,
\end{aligned}
\end{equation*}
and in particular $ u \in \W^{\alpha,p}$ a.s.
\end{proof}
\begin{Proposition}
For every $\alpha < \frac12$, $S(t)\u \in C_t\C^\alpha$ a.s. Moreover,
\begin{equation*}
\sup_{t>0} \norm{e^{\frac t8}S(t) \u}_{\C^\alpha} <+\infty
\end{equation*}
a.s. In particular, $\u\in X^\alpha$ a.s.
\end{Proposition}
\begin{proof}
For a test function $\f$, we have that
$$\resizebox{\hsize}{!}
{$\displaystyle\E\dual{S(t)\u}{\f}\dual{S(s)\u}{\f} = \dual{\mathcal L S(t)^\ast \f}{S(s)^\ast \f} = \dual{S(s)\mathcal L\f}{S(t)\f} = e^{-\frac{t+s}2} \bar\gamma(t,s),$}
$$
where 
$$\mathcal L \vec{f}{f_t} = \vec{(1 + \Delta^2)^{-1} f}{f_t}.$$
Therefore, we have $\bar \gamma(t,s)[\f] \lesssim \norm{f}_{H^{-2}}^2 + \norm{f_t}_{L^2}^2$ and $\Lip(\bar \gamma(\cdot,\cdot)[\f])\lesssim \norm{f}_{L^2}^2 + \norm{f_t}_{H^2}^2$, 
so we can conclude the proof exactly in the same way as in Proposition \ref{stickX}.
\end{proof}

However, we are not interested in $\mu$, but in the \emph{Gibbs measure} $\rho$, which formally is given by 
\begin{align*}
\d\rho(\u) &= Z^{-1}\exp\Big(-\frac14\int u^4-\frac12 \int u^2 - \frac12 \int (\Delta u)^2\Big)\exp\Big(-\frac12\int u_t^2\Big)\d u \d u_t \\
&= Z^{-1}\exp\Big(-\frac14\int u^4\Big)\d\mu(\u),
\end{align*}
where $Z$ is the normalisation factor.
\begin{Proposition} \label{convergence}
The function $F(\u) := \exp\big(-\frac14\int u^4\big)$ belongs to $L^\infty(\mu)$ and $\norm{F}_{L^\infty(\mu)} \le 1$.
Moreover, if $F_N(\u) := \exp\big(-\frac14\int (P_{\le N}u)^4\big)$, then $\norm{F_N}_{L^\infty(\mu)}\le1$ and
$F_N\to F$ in $L^p(\mu)$ for every $1\le p <+\infty$.

 In particular, the probability measures 
$\rho_N := Z_N^{-1}F_N\mu$, $\rho:= Z^{-1}F\mu$ are well defined, absolutely continuous with respect to $\mu$, and, for every set $E$, $\rho_N(E) \to \rho(E)$.
\end{Proposition}
\begin{proof}
By Proposition \ref{muholder}, $\int u^4 < +\infty$ $\mu$-a.s., so $F, F_N$ are well defined $\mu$-a.s. 
Moreover, since $\int f^4 \ge 0$ for every $f$, we have that $F,F_N \le 1$. Again by Proposition \ref{muholder}, 
 $P_{\le N} u \to u$ in $L^4$, so up to subsequences, $F_N \to F$ $\mu$-a.s., therefore $\int|F_N-F|^p\d\mu \to 0$ by dominated convergence. 
\end{proof}
\section{Local and global well posedness}
In this section, we will show local and global well posedness in $X^\alpha$ for the equations \eqref{bvec}, \eqref{bvecN}, relying onto the probabilistic estimates of the previous section and the \emph{Da Prato-Debussche} trick.

Local well posedness will follow by a standard Banach fixed point argument. For global well posedness, following \cite{bt14} we estimate an appropriate energy for the remainder $\v$, and we combine this argument with an integration by parts trick from \cite{op16}.
\subsection{Local well posedness}
\begin{Proposition} \label{lwp}
The equations 
\begin{align}
\v(t) &=  S(t) \v_0 - \int_0^t S(t-t') \vec{0}{(S(t')\u_0 + \stick_{t'}(\xi) + v)^3}\d t', \label{iterable}\\
\v_N(t) &=  S(t) \v_0 - \int_0^t S(t-t') P_{\le N} \vec{0}{P_{\le N}(S(t')\u_0 + \stick_{t'}(\xi) + v_N)^3}\d t' \label{iterableN},
\end{align}
where $\v_0 \in \H^2$ and $u_0\in X^\alpha$, are locally well-posed. 

More precisely, there exists $T(\norm{\v_0}_{\H^2},\norm{\u_0}_{X^\alpha},\norm{\Stick}_{C([0,1];\C^0)})>0$ such that there exists a unique solution $\v(t;\v_0,\u_0,\xi) \in C_t\H^2$, defined 
on a maximal interval (respectively) $[0,T^*(\v_0,\u_0,\xi))$, $[0,T_N^*(\v_0,\u_0,\xi))$, with $T^*,T_N^* > T$.

Moreover, if $T^*<+\infty$, then the following blowup condition holds
\begin{equation}\label{blowup}
\lim_{t\to T^*} \norm{\v(t)}_{\H^2} = +\infty.
\end{equation}
Respectively, if $T_N^* < +\infty$, we have 
$$\lim_{t\to T_N^*} \norm{\v_N(t)}_{\H^2} = +\infty.
 $$
\end{Proposition}
\begin{proof}Consider the map $\Gamma = \Gamma_{\v_0,\u_0,\stick}$ given respectively by
\begin{align*}
\Gamma(\v) &= S(t)\v_0 -  \int_0^t S(t-t')\vec{0}{(S(t')\u_0 + \stick_{t'}(\xi) + v)^3}\d t',\\
\Gamma(\v_N) &= S(t)\v_0 -  \int_0^t S(t-t')P_{\le N}\vec{0}{P_{\le N}(S(t')\u_0 + \stick_{t'}(\xi) + v_N)^3}\d t'.
\end{align*}
We want to show that 
for some universal $C>0$, $R=2C(1+\norm{\v_0}_{\H^2})$, $T=T(\norm{\v_0}_{\H^2},\norm{\u_0}_{X^\alpha},\norm{\stick}_{C([0,1];\C^0)})>0$, this map is a contraction on \linebreak $B_R\subseteq C([0,T];\H^2)$.
By the uniform boundedness of $S(t)$ and $P_{\le N}$ as operators $\H^2\to\H^2$, we have that, if $\v\in B_R$, 
\begin{equation} \label{contraction}
\begin{aligned}
\norm{\Gamma(\v(t))}_{\H^2} &\lesssim \norm{\v_0}_{\H^2} + T\sup_{0\le t\le T} \norm{\vec{0}{(S(t)\u_0 + \<1>(\xi) + v)^3}}_{\H^2} \\
&= \norm{\v_0}_{\H^2} + T\sup_{0\le t\le T} \norm{(S(t)\u_0 + \<1>(\xi) + v)^3}_{L^2}\\
%&\lesssim  \norm{\v_0}_{\H^2} + T\sup_{0\le t\le T} \norm{S(t)\u_0 + \stick(\xi) + v}_{L^\infty}^3 \\
&\lesssim \norm{\v_0}_{\H^2} + T(\norm{\u_0}_{X^\alpha}^3 + \norm{\stick(\xi)}_{C([0,T];\C^0)}^3 + \sup_{0\le t\le T}\norm{\v}_{\H^2}^3). \\
& \lesssim \norm{\v_0}_{\H^2} + T(\norm{\u_0}_{X^\alpha}^3 + \norm{\stick(\xi)}_{C([0,T];\C^0)}^3) + TR^3,
\end{aligned}
\end{equation}
where we just used the Sobolev embedding $\norm{v}_{L^\infty} \lesssim \norm{v}_{H^2}$.

Therefore, if $C$ is the implicit constant in this inequality, for $T$ small enough 
($T < \frac{R}{2C(\norm{\u_0}_{X^\alpha}^3 + \norm{\stick(\xi)}_{C([0,1];\C^0)}^3 +R^3)} \wedge 1$), 
$\Gamma$ maps $B_R$ into itself. 
Proceeding similarly,
\begin{align*}
&\norm{\Gamma(\v(t))-\Gamma(\w(t))}_{\H^2} \\
%&\lesssim T\sup_{0\le t\le T} \norm{\vec{0}{(S(t)\u_0 + \<1>(\xi) + v)^3 - (S(t)\u_0 + \<1>(\xi) + w)^3}}_{\H^2} \\
%& = T \sup_{0\le t\le T} \norm{(S(t)\u_0 + \<1>(\xi) + v)^3 - (S(t)\u_0 + \<1>(\xi) + w)^3}_{L^2} \\
%&\lesssim T \sup_{0\le t\le T} (\norm{S(t)\u_0 + \<1>(\xi) + v}_{L^\infty}^2 + \norm{S(t)\u_0 + \<1>(\xi) + w}_{L^\infty}^2) \norm{v-w}_{L^\infty}\\
& \lesssim T(\norm{\u_0}_{X^\alpha}^2+\norm{\stick(\xi)}_{C([0,T];\C^0)}^2 + R^2)\norm{\v-\w}_{\H^2}.
\end{align*}
Therefore, for $T$ small enough ($T \le \frac1{2C'(\norm{\u_0}_{X^\alpha}^2+\norm{\stick(\xi)}_{C([0,1];\C^0)}^2 + R^2)} \wedge 1$, where $C'$ is the implicit constant in the inequality), 
$$\norm{\Gamma(\v(t))-\Gamma(\w(t))}_{\H^2} \le \frac12\norm{\v-\w}_{\H^2},$$
so $\Gamma$ is a contraction.
This implies that the equations \eqref{iterable}, \eqref{iterableN} have a unique solution in $B_R$ up to time $T(\norm{\v_0}_{\H^2},\norm{\u_0}_{X^\alpha},\norm{\stick}_{C([0,1];\C^0)})$.

We notice that if $\v$ solves \eqref{iterable}, then 
\begin{equation} \label{iterated}
\begin{aligned}
&\v(t+s) \\
&= S(t+s)\v_0 + \int_0^{t+s} S(t+s-t')\vec{0}{(S(t')\u_0 + \<1>_{t'}(\xi) + v)^3}\d t'\\
	&\begin{multlined} 
		=S(t)\left(S(s)\v_0 + \int_0^s S(s-t')\vec{0}{(S(t')\u_0 + \<1>_{t'}(\xi) + v)^3}\d t'\right) \\
		+ \int_s^{t+s} S(t+s-t')\vec{0}{(S(t')\u_0 + \<1>_{t'}(\xi) + v)^3}\d t'
		\end{multlined}\\
	&=S(t)\v(s) + \int_0^t S(t-t')\vec{0}{(S(t')(S(s)\u_0)+\stick_{s+t'}(\xi) + v(s+t'))^3} \d t',
\end{aligned}
\end{equation}
and similarly, if $\v_N$ solves \eqref{iterableN}, then
\begin{equation}\label{iteratedN}
\resizebox{1.0\hsize}{!}{$\displaystyle
\begin{aligned}
&\v_N(t+s) \\
= &S(t)\v(s) + \int_0^t S(t-t')P_{\le N}\vec{0}{P_{\le N}(S(t')(S(s)\u_0)+\stick_{s+t'}(\xi) + v_N(s+t'))^3} \d t'.
\end{aligned}$}
\end{equation}
Since $\v(s),\v_N(s) \in \H^2$, $S(s)\u_0 \in X^\alpha$ and $\stick_{s+\cdot}(\xi) \in C([0,1];\C^0)$, we can repeat the same
contraction argument on $\Gamma_{\v(s),S(s)\u_0,\stick_{s+\cdot}}$, and we obtain that  \eqref{iterated} 
and \eqref{iteratedN} have a unique solution on the interval 
$$[0,T(\norm{\v_0}_{\H^2},\norm{\u_0}_{X^\alpha},\norm{\stick}_{C([s,s+1];\C^0)})].$$ 
To show uniqueness up to time $T^*$ or $T_N^*$, suppose we have two different solutions $\v_1,\v_2$.
Let $s:=\inf\{t|\v_1(t)\neq\v_2(t)\}$. Then we have $\v_1(s)=\v_2(s)$, and both $\v_1(s+t)$ and $\v_2(s+t)$
solve either \eqref{iterated} or \eqref{iteratedN}, so they have to be equal up to time $T(\norm{\v_0}_{\H^2},\norm{\u_0}_{X^\alpha},\norm{\stick}_{C([s,s+1];\C^0)})$, which is in contradiction with the definition of $s$.

To show the blowup condition \eqref{blowup}, suppose by contradiction that $v$ solves $\eqref{iterable}$ and
$\norm{\v(t)}_{\H^2} \le C$ for every $t<T^*$. Taking $$T=T(C,\norm{\u_0}_{X^\alpha},\norm{\stick}_{C([0,T^*+1];\C^0)}),$$
let $s:= T^*-\frac T2$. We clearly have $T^* \ge T$, so $s >0$. Then $\v(s+\cdot)$ solves \eqref{iterated}, and we can extend the solution up to time
$$T(\norm{\v(s)}_{\H^2},\norm{S(s)\u_0}_{X^\alpha},\norm{\stick}_{C([s,s+1];\C^0)}) \ge T.$$ 
Therefore,
we can extend $\v$ as solution of \eqref{iterable} up to time $s+T = T^* + \frac T2$, contradiction.

The same argument holds for solutions of \eqref{iterableN}.
\end{proof}

\begin{Proposition}[Continuity of the flow] \label{continuity}
Let $\v(\u_0)$ solve \eqref{iterable} in an interval $[0, T^*)$. Let $ T <T^*$. Then there exists a neighbourhood $U$ of $\u_0$ such that $\v(\tilde\u_0)$ is a solution of \eqref{iterable} in the interval $[0,T]$ for every $\tilde\u_0 \in U \subseteq X^\alpha$ and $\lim_{\norm{\tilde \u_0 - \u_0}_{X^\alpha}\to 0} \v(\tilde\u_0) = \v(\u_0)$ in $C([0,T];\H^2)$. The same holds for solutions of \eqref{iterableN}.
\end{Proposition}
\begin{proof}
We prove the result just for solutions of \eqref{iterable}, the case of \eqref{iterableN} is completely analogous.
We want to prove that 
$$\tau^*:=\sup\left\{T\middle| 
\begin{gathered}
\exists U\ni \u_0 \text{ open such that} \\ 
%\v(\cdot):U\to C([0,T];\H^2) \text{is continuous},\\
 \norm{\v(\tilde \u_0)}_{C([0,T];\H^2)} \le \norm{\v(\u_0)}_{C([0,T];\H^2)} +1\,  \forall \tilde\v_0 \in U,\\
 \lim_{\tilde \u_0 \to \u_0} \v(\tilde\u_0) = \v(\u_0)
 \end{gathered}\right\} \ge T^*.$$
By definition, we have that $\tau^*\ge 0$ (with $U = B_\frac12(\u_0)$ for $T=0$). Suppose by contradiction that $\tau^* < T^*$. Let 
$$\tau = \frac12 T(\norm{\v(\u_0)}_{C[0,\tau^* + \epsilon];\H^2} + 1, \norm{\u_0} + 1, \norm{\stick}_{C([0,\tau^* + \epsilon];\C^0)}),$$
where $T(a,b,c)$ is defined as in the proof of Proposition \ref{lwp}. Proceeding as in the proof of Proposition
\ref{lwp}, $\v(\tilde\u_0)(t + [(\tau^*-\tau)\vee 0])$ will satisfy \eqref{iterated} with $s= (\tau^*-\tau) \vee 0$. Let $U$ be the set corresponding to $T=s$ in the definition of $\tau^*$. By definition of $\tau$, proceeding as in the proof of Proposition \ref{lwp}, 
$$
%\resizebox{1.0\hsize}{!}{$\displaystyle
\begin{aligned}
&(\Gamma_{\tilde \u_0} \v) (t+s) \\
:=& S(t)\v(\tilde\u_0)(s) + \int_0^t S(t-t')\vec{0}{(S(t')(S(s)\tilde\u_0)+\stick_{s+t'}(\xi) + v(s+t'))^3} \d t'
\end{aligned}
%$}
$$
will be a contraction (with Lipschitz constant $\frac12$) in the ball
$$B_R:=B_{2C(2+\norm{\v(\u_0)}_{C([0,\tau^* + \epsilon];\H^2)})} \subseteq C([s,(s+2\tau) \wedge (\tau^*+\epsilon)]; \H^2)$$
for every $\tilde\u_0 \in U$.
Moreover, these solutions will satisfy
\begin{align*}
&\Gamma_{\u_0}(\v(\tilde\u_0))(s+t)-\v(\tilde\u_0)(s+t)\\
=&\Gamma_{\u_0}(\v(\tilde\u_0))(s+t)-\Gamma_{\tilde\u_0}(\v(\tilde\u_0))(s+t)\\
=&S(t)(\v(\u_0)(s)-\v(\tilde\u_0)(s))\\
&+\int_0^t S(t-t')\vec{0}{(S(t')(S(s)\u_0)+\stick_{s+t'}(\xi) + v(s+t'))^3} \d t' \\
&-\int_0^t S(t-t')\vec{0}{(S(t')(S(s)\tilde\u_0)+\stick_{s+t'}(\xi) + v(s+t'))^3} \d t',
\end{align*}
so proceeding as in the proof of Proposition \ref{lwp}, and recalling that $2\tau\le 1$,
\begin{align*}
&\norm{\v(\u_0) - \v(\tilde\u_0)}_{C([s,\tau^*+\epsilon \wedge \tau];\H^2)} \\
&\lesssim \norm{\Gamma_{\u_0}(\v(\tilde\u_0))(s+t)-\v(\tilde\u_0)(s+t)}\\
&\begin{multlined}
\lesssim \norm{\v(\u_0)(s)-\v(\tilde\u_0)(s)}_{\H^2} \\+ 
(\norm{\u_0}_{X^\alpha} + \norm{\tilde \u_0}_{X^\alpha} + \norm{\stick(\xi)}_{C([0,\tau^*+\epsilon];\C^0)} + R)^2\norm{\u_0-\tilde\u_0}_{X^\alpha}.
\end{multlined}
\end{align*}
Therefore, for $\norm{\u_0-\tilde\u_0}_{X^\alpha} \to 0$, by definition of $\tau^*$ we have that 
$$\norm{\v(\u_0)(s)-\v(\tilde\u_0)(s)}_{\H^2} \to 0$$
as well and so 
$$\norm{\v(\u_0) - \v(\tilde\u_0)}_{C([s,\tau^*+\epsilon];\H^2)} \to 0,$$
which implies that 
in a small neighbourhood  $V\cap U$ around $\u_0$, 
$$\norm{\v(\tilde\u_0)}_{C([0,\tau^* + \epsilon\wedge\tau];\H^2)} \le \norm{\v(\u_0)}_{C([0,\tau^* + \epsilon\wedge\tau];\H^2)} + 1$$ 
and 
$$\lim_{\tilde\u_0\to\u_0}\v(\tilde\u_0) = \v(\tilde\u)\text{ in }C([0,\tau^* + \epsilon\wedge\tau];\H^2),$$ 
which contradicts the maximality of $\tau^*$.
\end{proof}

\begin{Lemma} \label{Nstability}
Let $\v_0\in\H^2$, $\u_0\in X^\alpha$, and let $\v_N$ be the solution of \eqref{iterableN}. Suppose that for some $K\in\R$,
\begin{equation*}
\sup_{N} \sup_{0\le t \le \bar T} \norm{\v_N}_{\H^2} \le K < +\infty.
\end{equation*}
Then the solution $\v$ to \eqref{iterable} satisfies $T^* \ge \bar T$, 
$\sup_{0\le t \le \bar T} \norm{\v}_{\H^2} \le K$, and $\norm{\v(t) - \v_N(t)}_{C([0,\bar T];\H^2)} \to 0$ as $N \to \infty$.
\end{Lemma}
\begin{proof}
Let 
$$\Gamma(\v)(t) = S(t)\v_0 -  \int_0^t S(t-t')\vec{0}{(S(t')\u_0 + \stick_{t'}(\xi) + v)^3}\d t'.\\$$
In Proposition \ref{lwp} we have shown that $\Gamma$ is a contraction (with $\Lip(\Gamma) \le \frac12$) in $B_R \subset C([0,T];\H^2)$,
where 
$$R = C(\norm{\v_0}_{\H^2} +1) \text{ and }T=T(\norm{\v_0}_{\H^2},\norm{\u_0}_{X^\alpha},\norm{\stick}_{C([0,1];\C^0)}).$$
We have that 
\begin{equation*}
\begin{aligned}
&\Gamma(\v_N)(t) - \v_N(t) \\
%&=S(t)\v_0 -  \int_0^t S(t-t')\vec{0}{P_{\le N}(S(t')\u_0 + \stick(\xi) + v)^3}\d t' - \v_N(t)\\
=&~S(t)\v_0 -  \int_0^t S(t-t')P_{\le N}\vec{0}{(P_{\le N}(S(t')\u_0 + \stick_{t'}(\xi) + v_N))^3}\d t' - \v_N(t) \\
&-  \int_0^t S(t-t')P_{\le N}\vec{0}{(S(t')\u_0 + \stick_{t'}(\xi) + v_N)^3 - (P_{\le N}(S(t')\u_0 + \stick_{t'}(\xi) + v_N))^3}\d t' \\
&- \int_0^t S(t-t')P_{>N}\vec{0}{(S(t')\u_0 + \stick_{t'}(\xi) + v_N)^3}\d t' \\
=&-  \int_0^t S(t-t')P_{\le N}\vec{0}{(S(t')\u_0 + \stick_{t'}(\xi) + v_N)^3 - (P_{\le N}(S(t')\u_0 + \stick_{t'}(\xi) + v_N))^3}\d t' \\
&- \int_0^t S(t-t')P_{>N}\vec{0}{(S(t')\u_0 + \stick_{t'}(\xi) + v_N)^3}\d t',
\end{aligned}
\end{equation*}
so 
\begin{align*}
&~\norm{\Gamma(\v_N) - \v_N}_{C([0,T];\H^2)}\\
\lesssim 
&~T\sup_{0\le t \le T} \norm{P_{>N} (S(t')\u_0 + \stick(\xi) + v_N)^3}_{L^2} \\
&~+ T\sup_{0\le t \le T} \norm{P_{>N} (S(t')\u_0 + \stick(\xi) + v_N)}_{L^2}\norm{(S(t')\u_0 + \stick(\xi) + v_N)^2}_{C^\epsilon}.
\end{align*}
By Bernstein's inequalities, $\norm{P_{>N} w}_{L^2} \lesssim N^{-\alpha}\norm{w}_{H^\alpha}
\lesssim N^{-\alpha}\norm{\w}_{\C^\alpha}$. Therefore, for $\alpha < \frac12$,
\begin{align*}
\norm{\Gamma(\v_N) - \v_N}_{C([0,T];\H^2)} 
&\lesssim TN^{-\alpha} (\norm{\u_0}_{X^\alpha}^3 +\norm{\stick}_{C([0,T];\C^\alpha)}^3 + \norm{\v_N}_{C([0,T];\H^2)}^3)\\
& \lesssim TN^{-\alpha} (\norm{\u_0}_{X^\alpha}^3 +\norm{\stick}_{C([0,T];\C^\alpha)}^3 + R^3),
\end{align*}
therefore for $N$ big enough, since $\Gamma$ is a contraction,
then 
\begin{equation*}
\norm{\v-\v_N}_{C([0,T];\H^2)} \lesssim TN^{-\alpha} (\norm{\u_0}_{X^\alpha}^3 +\norm{\stick}_{C([0,T];\C^\alpha)}^3 + R^3) \to 0 \text{ as }N\to\infty.
\end{equation*}
Now let $\mathscr T:=\sup\{\tau > 0 $ s.t. $\norm{\v(t) - \v_N(t)}_{C([0,\tau];\H^2)} \to 0$ as $N \to \infty\}$. 
We just proved that $\mathscr T \ge T(\norm{\v_0}_{\H^2},\norm{\u_0}_{X^\alpha},\norm{\stick}_{C([0,1];\C^0)})$.
Suppose by contradiction that $\mathscr T < \bar T$ or that $\mathscr T = \bar T$ but $\norm{\v(t) - \v_N(t)}_{C([0,\mathscr T];\H^2)} \not\to 0$ as $N\to \infty$. Let 
$T:= T(K,\norm{\u_0}_{X^\alpha},\norm{\stick}_{C([0,\bar T +1];\C^0)})$. Let $s:= \mathscr T - \frac T2$. 
Since $s< \mathscr T$, we have that $\v_N(s) \to \v(s)$ in $\H^2$, so $\norm{\v(s)} \le K$. Then $\v$ solves 
\eqref{iterated} in (at least) the interval $[s,s+T]$ and $\v_N$ solves \eqref{iteratedN} in the same interval. 
Redefining 
\begin{equation*}
\Gamma(\v):=S(t)\v(s) + \int_0^t S(t-t')\vec{0}{(S(t')(S(s)\u_0)+\stick_{s+t'}(\xi) + v(t'))^3} \d t',
\end{equation*}
and proceeding as before, we have that 
\begin{equation*}
\norm{\Gamma(\v_N)-\v_N}_{C([s,s+T];\H^2)} \lesssim TN^{-\alpha} (\norm{\u_0}_{X^\alpha}^3 +\norm{\stick}_{C([0,\bar T+1];\C^\alpha)}^3 + R^3),
\end{equation*}
and since $\Gamma$ is a contraction,
\begin{equation*}
\norm{\v-\v_N}_{C([s,s+T];\H^2)} \lesssim TN^{-\alpha} (\norm{\u_0}_{X^\alpha}^3 +\norm{\stick}_{C([0,\bar T+1];\C^\alpha)}^3 + R^3),
\end{equation*}
so we have that $\v_N\to \v$ uniformly in the interval $[s,s+T]$. Joining this with convergence in the interval
$[0,s]$, we have that 
\begin{equation*}
\norm{\v(t) - \v_N(t)}_{C([0,s+T];\H^2)} =\norm{\v(t) - \v_N(t)}_{C([0,\mathscr T + \frac T2];\H^2)} \to 0,
\end{equation*}
which is in contradiction with the definition of $\mathscr T$.
\end{proof}
\subsection{Global well posedness}
Consider the energy 
\begin{equation*}
E(\v) := \frac 12 \int v_t^2 +\frac12 \int v^2 + \frac12 \int (\Delta v)^2 + \frac14 \int v^4 + \frac18 \int(v+v_t)^2.
\end{equation*}
\begin{Proposition}\label{energy_estimate}
For every $0<\alpha<\frac12$, there exist $c > 0$ such that for every solution $\v_N$ of \eqref{iterableN} we have
\begin{equation*}
E(\v_N(t)) \lesssim_\alpha e^{-ct} E(\v_0) + \Big(1+  \norm{\u_0}_{X^\alpha}^\frac8\alpha + \norm{\stick_t}_{\C^\alpha}^4 + \int_0^t e^{-c(t-t')}\norm{\stick_{t'}}_{\C^\alpha}^\frac8\alpha \d t'\Big).
\end{equation*}
\end{Proposition}\noindent
Together with Lemma \ref{Nstability}, this implies that 
\begin{Corollary} \label{gwp}
Let $\v$ be a solution of \eqref{veqn} and let $\v_N$ be a solution of \eqref{veqnN}. 
Then for every $0< \alpha <\frac12$ and for every $N>1$ we have that 
\begin{equation*}
\norm{\v_N}_{\H^2}^2 \lesssim_\alpha \Big(1+  \norm{\u_0}_{X^\alpha}^\frac8\alpha + \norm{\stick_t}_{\C^\alpha}^4 + \int_0^t e^{-c(t-t')}\norm{\stick_{t'}}_{\C^\alpha}^\frac8\alpha \d t'\Big) 
< +\infty \text{ a.s.}.
\end{equation*}
Moreover, $\v$ is a.s.\ defined on the half-line $[0,+\infty)$, it satisfies the same estimate 
\begin{equation} \label{vgrowth}
\norm{\v}_{\H^2}^2 \lesssim_\alpha \Big(1+  \norm{\u_0}_{X^\alpha}^\frac8\alpha + \norm{\stick_t}_{\C^\alpha}^4 + \int_0^t e^{-c(t-t')}\norm{\stick_{t'}}_{\C^\alpha}^\frac8\alpha \d t'\Big) 
\end{equation}
 and for every $T < +\infty$,
\begin{equation*}
\norm{\v-\v_N}_{C([0,T];\H^2)} \to 0 \text{ a.s.}.
\end{equation*}
\end{Corollary}
\begin{Remark} \label{smoothN}
Any solution $\v$ of \eqref{veqnN} actually belongs to $C^1([0,T^*);\C^\infty)$. Indeed, for any $t\le T<T^*$, proceeding like in \eqref{contraction},
\begin{align*}
\norm{\v_N(t)}_{\H^{2+s}} &= \norm{\jap{\nabla}^s \int_0^t P_{\le N}S(t-t') \vec{0}{(P_{\le N}(S(t)\u_0 + \stick(\xi) + v_N))^3}\d t'}_{\H^2} \\
&\lesssim T \sup_{0\le t\le T} \norm{\jap{\nabla}^2 P_{\le N} (S(t)\u_0 + \stick(\xi) + v_N)^3}_{L^2} \\
&\lesssim T N^2 \sup_{0\le t\le T} \norm{(S(t)\u_0 + \stick(\xi) + v_N)^3}_{L^2} < +\infty,
\end{align*}
where we just used that $\norm{\jap{\nabla}^s P_{\le N}}_{L^2\to L^2} \lesssim N^s$. Similarly, 
\begin{align*}
\norm{\partial_t \v(t)}_{\H^{s}} 
%&= \norm{
%\begin{multlined}
%\jap{\nabla}^s \int_0^t P_{\le N}(\partial_t S(t-t')) \vec{0}{(S(t')\u_0 + \stick(\xi) + v)^3}\d t' \\
%+ \jap{\nabla}^s\vec{0}{(S(t)\u_0 + P_{\le N} \stick_t(\xi) + v(t))^3}
%\end{multlined}
%}_{\H^0} \\
%&\lesssim T \sup_{0\le t\le T} \norm{\jap{\nabla}^s P_{\le N} (S(t)\u_0 + \stick(\xi) + v)^3}_{L^2} \\
&\lesssim T N^s \sup_{0\le t\le T} \norm{(S(t)\u_0 + \stick(\xi) + v)^3}_{L^2} < +\infty.
\end{align*}
Proceeding in this way, it is actually possible to show that $\v \in C^\infty_t([0,T^*);\C^\infty)$, however, we will never need
more regularity than $C^1$ (in time).

In the remainder of this subsection, in order to make the notation less cumbersome, we will omit the subscript $N$ from $\v_N$ whenever it is not relevant in the analysis.
\end{Remark}
\begin{Lemma} \label{energy_derivative}
If $\v$ solves \eqref{veqnN}, then
\begin{align} 
\partial_t E(\v)=& -\frac14\Big(3\int v_t^2 + \int v^2 + \int (\Delta v)^2 + \int v^4\Big) \label{-E}\\
&+\partial_t\Big(\frac18 \int v_t^2 \Big) \label{derivative}\\
&- 3\int v_t v^2 (\stick_t(\xi) + S(t)\u_0) \label{main_term} \\
&-\int (v_t+\frac14 v) [(S(t)\u_0)  3v(\stick_t(\xi) +S(t)\u_0)^2 + (\stick_t(\xi) +S(t)\u_0)^3] \label{lots}\\
&- \frac34\int v^3(\stick_t(\xi) +S(t)\u_0) \label{lots2}\\
&+\int(v_t+\frac14v)P_{>N}(v+\stick_t(\xi)+S(t)\u_0)^3 \label{zeroN}
\end{align}
\end{Lemma}
\begin{proof}
By Remark \ref{smoothN}, $E(\v)$ is differentiable, and moreover $\v$ satisfies
\begin{equation*}
\partial_t \begin{pmatrix} v \\ v_t \end{pmatrix} = 
-\begin{pmatrix} 0 & -1 \\ 1 + \Delta^2 & 1\end{pmatrix}\begin{pmatrix} v \\ v_t \end{pmatrix} - \begin{pmatrix} 0 \\ P_{\le N}(v+\stick_t(\xi)+S(t)\u_0)^3 \end{pmatrix}. \\
\end{equation*}
The formula follows by computing $\partial_t E$, and using the equation to substitute the term $v_{tt}$.
%Therefore, by exchanging time derivatives with integrals and using the equation, we have that 
%\begin{align*}
%&\partial_t E(\v)\\
%=& \int v_t(v_{tt} + v + \Delta^2 v + v^3) \\
%&+ \frac18\partial_t\Big(\int v^2 + \int v_t^2\Big) + \frac14 \int v_t^2 + \frac14 \int vv_{tt}\\
%=&-\int v_t^2 +\int v_t(v^3-P_{\le N}(v+\stick_t(\xi)+S(t)\u_0)^3)\\
%&+ \frac18\partial_t\Big(\int v^2 + \int v_t^2\Big) + \frac14 \int v_t^2\\
%&-\frac14 \int vv_t -\frac14\Big(\int v^2 + \int (\Delta v)^2 + \int v P_{\le N}(v+\stick_t(\xi)+S(t)\u_0)^3\Big) \\ 
%=& -\int v_t^2 +\int v_t(v^3-(v+\stick_t(\xi)+S(t)\u_0)^3) \\
%&+ \int v_t P_{> N}(v+\stick_t(\xi)+S(t)\u_0)^3)\\
%&+ \frac18\partial_t\Big(\int v^2 + \int v_t^2\Big) + \frac14 \int v_t^2\\
%&-\frac18 \partial_t\Big(\int v_t^2\Big)-\frac14\Big(\int v^2 + \int (\Delta v)^2 + \int v(v+\stick_t(\xi)+S(t)\u_0)^3\Big)\\
%&+\frac14  \int v P_{>N}(v+\stick_t(\xi)+S(t)\u_0)^3.\\
%\stackrel{\text{rearranging}}=&
%-\frac14\Big(3\int v_t^2 + \int v^2 + \int (\Delta v)^2 + \int v(v+\stick_t(\xi)+S(t)\u_0)^3\Big)\\
%&+\partial_t\Big(\frac18 \int v_t^2\Big)\\
%&+\int v_t(v^3-(v+\stick_t(\xi)+S(t)\u_0)^3)\\
%&+\int(v_t+\frac14v)P_{>N}(v+\stick_t(\xi)+S(t)\u_0)^3,
%\end{align*}
%and the claimed identity follows from expanding the cubes.
\end{proof}
\begin{Lemma} \label{P>N=0}
If $\v$ solves \eqref{veqnN}, then 
\begin{equation*}
\eqref{zeroN} = 0.
\end{equation*}
\end{Lemma}
\begin{proof}
If $\v$ solves \eqref{veqnN}, then we can write $\v$ in the form $\v = P_{\le N} \w$ for some $\w$, therefore $P_{>N}\v = 0.$ Therefore,
\begin{align*}
\eqref{zeroN} &= \int(v_t+\frac14v)P_{>N}(v+\stick_t(\xi)+S(t)\u_0)^3 \\
&= \int P_{\le N}(v_t+\frac14v)P_{>N}(v+\stick_t(\xi)+S(t)\u_0)^3 = 0.
\end{align*}
\end{proof}
\begin{Lemma} \label{lots_estimate}
If $\v$ solves \eqref{veqnN}, then for every $0<\alpha<\frac12$, 
\begin{equation*}
\eqref{lots} \lesssim E^\frac34 (\norm{\u_0}_{X^\alpha} + \norm{\stick_t(\xi)}_{\C^\alpha})^2 + E^\frac12 (\norm{\u_0}_{X^\alpha} + \norm{\stick_t(\xi)}_{\C^\alpha})^3
\end{equation*}
\end{Lemma}
\begin{proof}
By H\"older, we have that 
\begin{equation*}
\int (v_t+\frac14 v) 3v(\stick_t(\xi) +S(t)\u_0)^2 \lesssim (\norm{v}_{L^2}+\norm{v_t}_{L^2})\norm{v}_{L^4}
\norm{(\stick_t(\xi) +S(t)\u_0)^2}_{L^4},
\end{equation*}
so by noticing that $\norm{v_t}_{L^2} \lesssim E^\frac12$, $\norm{v}_{L^2} \lesssim E^\frac12$,
$\norm{v}_{L^4} \lesssim E^\frac 14$, and $$\norm{(\stick_t(\xi) +S(t)\u_0)^2}_{L^4} \lesssim \norm{(\stick_t(\xi) +S(t)\u_0)}_{C^\alpha}^2\lesssim (\norm{\u_0}_{X^\alpha} + \norm{\stick_t(\xi)}_{\C^\alpha})^2,$$
we have that 
\begin{equation*}
\int (v_t+\frac14 v) 3v(\stick +S(t)\u_0)^2 \lesssim E^\frac34 (\norm{\u_0}_{X^\alpha} + \norm{\stick_t(\xi)}_{\C^\alpha})^2.
\end{equation*}
Proceeding similarly, 
\begin{align*}
\int (v_t+\frac14 v) (\stick_t(\xi) +S(t)\u_0)^3 &\lesssim (\norm{v}_{L^2}+\norm{v_t}_{L^2})\norm{(\stick_t(\xi) +S(t)\u_0)^3}_{L^2} \\
&\lesssim E^\frac12 (\norm{\u_0}_{X^\alpha} + \norm{\stick_t(\xi)}_{\C^\alpha})^3.
\end{align*}
\end{proof}
\begin{Lemma} \label{lots2_estimate}
If $\v$ solves \eqref{veqnN}, then for every $0<\alpha<\frac12$, 
\begin{equation*}
\eqref{lots2} \lesssim E^\frac34 (\norm{\u_0}_{X^\alpha} + \norm{\stick_t(\xi)}_{\C^\alpha})^2 
\end{equation*}
\end{Lemma}
\begin{proof}
By H\"older, 
$$\eqref{lots2} \les \norm{v}_{L^4}^3 \norm{S(t)\u_0 + \stick_t(\xi)}_{L^4} \les E^\frac34 (\norm{\u_0}_{X^\alpha} + \norm{\stick_t(\xi)}_{\C^\alpha})^2. $$
\end{proof}

\begin{Lemma} \label{intbyparts}
If $\v$ solves \eqref{veqnN}, then
\begin{equation} \label{part_diff}
\eqref{main_term} = -\partial_t\Big(\int v^3 (\stick_t(\xi) + S(t)\u_0) \Big) + \int  v^3 \partial_t (\stick_t(\xi) + S(t)\u_0),
\end{equation}
and for every $0< \alpha<\frac12$,
\begin{equation*}
\int  v^3 \partial_t (\stick_t(\xi) + S(t)\u_0) \lesssim E^{1-\frac\alpha8} (\norm{\u_0}_{X^\alpha} + \norm{\stick_t(\xi)}_{\C^\alpha}).
\end{equation*}
%\todo{Cite integration by parts trick - Hiro + Oana?}
\end{Lemma}
\begin{proof}
\eqref{part_diff} follows just from Leibniz rule. In order to prove the estimate, notice that 
$\norm{v}_{L^4} \lesssim E^\frac14$, and $\norm{v}_{H^2} \lesssim E^\frac12$. Therefore, by 
H\"older and fractional Leibniz respectively, 
\begin{equation*}
\left\{
\begin{aligned}
\norm{v^3}_{L^\frac43} &\lesssim E^\frac34\\
\norm{v^3}_{W^{2,1}} & \lesssim E.
\end{aligned}
\right.
\end{equation*}
Therefore, by interpolation (Gagliardo - Nirenberg), if $\frac1p = (1-\frac{\alpha}2) + \frac{\alpha}{2}\cdot\frac 34 = 1-\frac\alpha8$, then 
$\norm{v^3}_{W^{2-\alpha,p}} \lesssim E^{(1-\frac{\alpha}2) + \frac{\alpha}{2}\cdot\frac 34} = E^{1-\frac\alpha8}$.
Hence
\begin{align*}
\int  v^3 \partial_t (\stick_t(\xi) + S(t)\u_0) &\lesssim \norm{v^3}_{{W^{2-\alpha,p}}} \norm{\partial_t (\stick_t(\xi) + S(t)\u_0)}_{W^{\alpha-2,p'}} \\
& \lesssim \norm{v^3}_{{W^{2-\alpha,p}}} \norm{\stick_t(\xi) + S(t)\u_0}_{\W^{\alpha,p'}}\\
&\lesssim E^{1-\frac\alpha8} (\norm{\u_0}_{X^\alpha} + \norm{\stick_t(\xi)}_{\C^\alpha}).
\end{align*}
\end{proof}
\begin{proof}[Proof of Proposition \ref{energy_estimate}]
Let $F(\v) := E(\v) - \frac 18\int v_t^2 + \int v^3(\stick_t(\xi) + S(t)\u_0)$. By H\"older and Young's inequalities,
\begin{align*}
\Big|\int v^3(\stick_t(\xi) + S(t)\u_0)\Big| &\le \norm{v}_{L^4}^3 \norm{(\stick_t(\xi) + S(t)\u_0)}_{L^4} \\
&\le E^\frac34(\norm{\stick_t(\xi)}_{C^\alpha}+\norm{\u_0}_{X^\alpha})\\
& \le \frac14 E + \frac{27}4(\norm{\stick_t(\xi)}_{C^\alpha}+\norm{\u_0}_{X^\alpha})^4.
\end{align*}
Therefore,
\begin{gather}
F \le \frac54 E + \frac{27}4(\norm{\stick_t(\xi)}_{C^\alpha}+\norm{\u_0}_{X^\alpha})^4 \label{F<E},\\
E \le 2 F + \frac{27}2(\norm{\stick_t(\xi)}_{C^\alpha}+\norm{\u_0}_{X^\alpha})^4 \label{E<F}.
\end{gather} 
Using Lemma \ref{energy_derivative} and \eqref{part_diff}, we have that 
\begin{align*}
\partial_t F =& -\frac14\Big(3\int v_t^2 + \int v^2 + \int (\Delta v)^2 + \int v^4\Big) \\
&+\int  v^3 \partial_t (\stick_t(\xi) + S(t)\u_0)\\
&-\int (v_t+\frac14 v) [3v(\stick_t(\xi) +S(t)\u_0)^2 + (\stick_t(\xi) +S(t)\u_0)^3] \\
&+\int(v_t+\frac14v)P_{>N}(v+\stick_t(\xi)+S(t)\u_0)^3.
\end{align*}
Therefore, using Lemma \ref{intbyparts}, Lemma \ref{lots_estimate}, Lemma \ref{lots2_estimate}, Lemma \ref{P>N=0}, 
Young's inequality and \eqref{F<E}, for some constant $C$ (that can change line by line) we have
\begin{align*}
\partial_t F \le &- \frac12 E\\
			&+E^{1-\frac\alpha8} (\norm{\u_0}_{X^\alpha} + \norm{\stick_t(\xi)}_{\C^\alpha})\\
			&+ E^\frac34 (\norm{\u_0}_{X^\alpha} + \norm{\stick_t(\xi)}_{\C^\alpha})^2 + E^\frac12 (\norm{\u_0}_{X^\alpha} + \norm{\stick_t(\xi)}_{\C^\alpha})^3\\
\le & -\frac12 E + \frac14 E \\
& +C\Big[(\norm{\u_0}_{X^\alpha} + \norm{\stick_t(\xi)}_{\C^\alpha})^\frac8\alpha +(\norm{\u_0}_{X^\alpha} + \norm{\stick_t(\xi)}_{\C^\alpha})^8 \\
&  \phantom{+C\Big[}+(\norm{\u_0}_{X^\alpha} + \norm{\stick_t(\xi)}_{\C^\alpha})^6\Big]\\
\le& -\frac12 E + C\Big(1 + \norm{\u_0}_{X^\alpha}^\frac8\alpha + \norm{\stick_t(\xi)}_{\C^\alpha}^\frac8\alpha\Big)\\
\le& -\frac25 F + \frac{27}{10} (\norm{\stick_t(\xi)}_{\C^\alpha}+\norm{\u_0}_{X^\alpha})^4 + C\Big(1 + \norm{\u_0}_{X^\alpha}^\frac8\alpha + \norm{\stick_t(\xi)}_{\C^\alpha}^\frac8\alpha\Big)\\
\le& -\frac25 F + C\Big(1 + \norm{\u_0}_{X^\alpha}^\frac8\alpha + \norm{\stick_t(\xi)}_{\C^\alpha}^\frac8\alpha\Big).
\end{align*}
Therefore, by Gronwall, if $c:=\frac25$, for some other constant $C$ we have
\begin{equation*}
F(\v(t)) \le e^{-ct}F(\v_0) + C\Big(1 + \norm{\u_0}_{X^\alpha}^\frac8\alpha + \int_0^t e^{-c(t-t')} \norm{\stick_{t'}(\xi)}_{\C^\alpha}^\frac8\alpha \d t'\Big).
\end{equation*}
Hence, using \eqref{E<F} and \eqref{F<E},
\begin{align*}
&E(\v(t)) \\
\lesssim& F(\v(t)) + (\norm{\stick_t(\xi)}_{\C^\alpha}+\norm{\u_0}_{X^\alpha})^4 \\
\lesssim& e^{-ct}F(\v(0)) + 1 + \norm{\u_0}_{X^\alpha}^\frac8\alpha + \int_0^t e^{-c(t-t')} \norm{\stick_{t'}(\xi)}_{\C^\alpha}^\frac8\alpha \d t' + \norm{\stick_t(\xi)}_{\C^\alpha}^4+\norm{\u_0}_{X^\alpha}^4\\
\lesssim& e^{-ct}(\norm{\u_0}_{X^\alpha}^4) + 1 + \norm{\u_0}_{X^\alpha}^\frac8\alpha + \norm{\stick_t(\xi)}_{\C^\alpha}^4 + \int_0^t e^{-c(t-t')} \norm{\stick_{t'}(\xi)}_{\C^\alpha}^\frac8\alpha \d t'\\
\lesssim& 1 + \norm{\u_0}_{X^\alpha}^\frac8\alpha + \norm{\Stick_{t}(\xi)}_{\C^\alpha}^4 + \int_0^t e^{-c(t-t')} \norm{\stick_{t'}(\xi)}_{\C^\alpha}^\frac8\alpha \d t'.
\end{align*}
\end{proof}
\section{Invariance}
The goal of this section is showing that the flow of \eqref{bvec} is a stochastic flow which satisfies the semigroup property, and proceed to prove that the measure $\rho$ is invariant for the flow of \eqref{bvec}. 
Even if we will not use it explicitly in the following, the semigroup property ensures that the maps on Borel functions
$$F \mapsto P_t F := \E[F(\Phi_t(\cdot;\xi))] $$
define a Markov semigroup, to which we can apply the usual theory for stationary measures.

Recall that, if $\u_0 \in X^ \alpha$, the flow 
of \eqref{bvec} at time $t$ with initial data $\u_0$ is defined as 
\begin{equation*}
\Phi_t(\u_0;\xi) = S(t) \u_0 + \stick_t(\xi) + \v(\u_0,\xi;t),
\end{equation*}
where $\v$ solves \eqref{veqn}.
\begin{Proposition}
The map $\Phi$ satisfies the semigroup property, i.e.\ for every $F$ measurable and bounded,
\begin{equation*}
\E[F(\Phi_{t+s}(\u_0;\xi))] = \E[F(\Phi_s(\Phi_t(\u_0;\xi_1);\xi_2))],
\end{equation*}
where $\xi_1,\xi_2$ are two independent copies of space-time white noise. 
\end{Proposition}
\begin{proof}
%Notice that 
%\begin{align}
%&\phantom{=} \Phi_s(\Phi_t(\u_0;\xi_1),\xi_2)\notag \\
%&= S(s) \Phi_t(\u_0;\xi_1) + \stick_s(\xi_2) + \v(\Phi_t(\u_0;\xi_1),\xi_2;s) \notag\\
%&= S(t+s)\u_0 + S(s) \stick_t(\xi_1) + S(s) \v(\u_0, \xi_1;t) + \stick_s(\xi_2) + \v(\Phi_t(\u_0;\xi_1),\xi_2;s)\notag\\
%&= S(t+s)\u_0 + S(s) \stick_t(\xi_1) + \stick_s(\xi_2) + S(s) \v(\u_0, \xi_1;t)+ \v(\Phi_t(\u_0;\xi_1),\xi_2;s).\label{Phit+s}
%\end{align}
Given $\xi_1,\xi_2$ independent copies of the white noise, let $\tilde \xi$ be defined by 
$$\langle{\tilde \xi},{\phi}\rangle := \dual{\1_{[0,t]}\xi_1}{\phi} + \dual{\1_{(t,+\infty)} \xi_2(\cdot - t)}{\phi}$$
for every test function $\phi$.
It is easy to see that $\tilde \xi$ satisfies the universal property
$$\E[|\langle{\tilde \xi},{f}\rangle|^2] = \norm{f}_{L^2}^2,$$ 
so it is a copy of space-time white noise. Moreover, from a direct computation, 
$$\Phi_s(\Phi_t(\u_0;\xi_1),\xi_2) = \Phi_{t+s}(\u_0;\tilde \xi), $$
and so for every $F$ measurable and bounded,
$$\E[F(\Phi_{t+s}(\u_0;\xi))] = \E[F(\Phi_s(\Phi_t(\u_0;\xi_1);\xi_2))].$$

\end{proof}

\begin{Proposition}\label{invarianceN}
Consider the flow given by 
\begin{equation}\label{flowN}
\Phi^N_t(\u_0;\xi) := S(t) \u_0 + \stick_t(\xi) + \v_N(\u_0;\xi),
\end{equation}
where $\v_N$ solves \eqref{veqnN}. Then the measure 
\begin{equation}\label{rhoN}
\d \rho_N(\u) :=\frac{1}{Z_N} \exp\Big( - \frac 14 \int (P_{\le N}u)^4 \Big) \d \mu (\u)
\end{equation}
is invariant for the the process associated to the flow $\Phi^N_t(\cdot;\xi)$, where $Z_N = \int  \exp\Big( - \frac 14 \int (P_{\le N}u)^4 \Big) \d \mu (\u)$ (so that $\rho_N$ is a probability measure).
\end{Proposition}
\begin{proof}
Let $\mathbf X$ be a random variable with law $\mu$, independent from $\xi$. Invariance of \eqref{rhoN} is equivalent to showing that 
$$\E\Big[F(\Phi^N_t(\mathbf X;\xi))\exp\Big( -\frac 14 \int (P_{\le N}\pi_1\mathbf X)^4 \Big)\Big] = \E\Big[F(\mathbf X)\exp\Big( -\frac 14 \int (P_{\le N}\pi_1 \mathbf X)^4 \Big)\Big]  $$
for every $F: X^\alpha \to \R$ continuous. Let $M \ge N$. By definition of $X^\alpha$, we have that 
$\lim_{M \to \infty} \norm{\u- P_M \u}_{X^\alpha} = 0$ for every $u \in X^{\alpha'}$, $\alpha' > \alpha$. Therefore, 
by Proposition \ref{stickX} and Proposition \ref{continuity}, one has that for every $t\ge 0$, 
$$ \lim_{M\to \infty} \norm{ \Phi^N_t (P_M\mathbf X;P_M\xi) - \Phi_t^N (\mathbf X;\xi)}_{X^\alpha} = 0.$$
Therefore, by dominated convergence, it is enough to prove that 
\begin{equation}\label{invarianceNM}
\begin{aligned}
&\E\Big[F(\Phi^N_t(P_M\mathbf X;P_M\xi))\exp\Big( -\frac 14 \int (P_{\le N}\pi_1 \mathbf X)^4 \Big)\Big] \\
=& \E\Big[F(P_{\le M}\mathbf X)\exp\Big( -\frac 14 \int (P_{\le N}\pi_1 \mathbf X)^4 \Big)\Big].  
\end{aligned}
\end{equation}
By \eqref{flowN}, it is easy to check that $\mathbf Y= (Y,Y_t)^T:= \Phi^N_t(\cdot;P_M\xi)$ solves the SDE
$$ \d \mathbf Y = \begin{pmatrix}
0 & 1 \\
-(1+\Delta^2) & -1
\end{pmatrix}
\mathbf Y - P_{\le N}\vec{0}{(P_{\le N} Y)^3}+ \vec{0}{\sqrt2 \d W_M}, $$
where $\d W_M:= P_M \xi$ is a space-time white noise on the finite dimensional space given by the image of the map $P_M$. 
Therefore, if we show that the measure $\tilde \rho$ defined on the image of $P_M$, 
$$\d \tilde \rho(u):= \exp\Big(-\frac 14 \int (P_{\le N} u)^4 - \frac12 \int |u|^2 - \frac12 \int |\Delta u|^2 - \frac12 \int |u_t|^2\Big) \d u \d u_t, $$
is invariant for the flow $\mathbf Y$, we get \eqref{invarianceNM}. Since $\mathbf Y$ solves an SDE with smooth coefficients, this is true if and only if $\tilde \rho$ solves 
the Fokker-Planck equation 
\begin{equation*}\label{Fokker-Planck}
-\div\left[\Big(\begin{pmatrix}
 0 & 1
 \\-(1+\Delta^2) & -1
\end{pmatrix} \vec{u}{u_t} - \vec{0}{u^3}\Big)\tilde \rho(u,u_t)\right] = 0,
\end{equation*}
which can be shown through a direct computation.
%We have that 
%\begin{align}
%& -\div\left[\Big(\begin{pmatrix}
% 0 & 1
% \\-(1+\Delta^2) & -1
%\end{pmatrix} \vec{u}{u_t}- \vec{0}{u^3}\Big)\right] \tilde \rho(u,u_t)  \nonumber \\
%&= \dim\big(\{P_Mu_t\}\big)  \tilde \rho(u,u_t) \nonumber\\
%&= (2M+1)^3  \tilde \rho(u,u_t),
%\end{align}
%\begin{align*}
%&\Big(-\begin{pmatrix}
% 0 & 1
% \\-(1+\Delta^2) & -1
%\end{pmatrix}\vec{u}{u_t} - \vec{0}{u^3}\Big)\cdot \nabla \tilde \rho(u,u_t) \\
%=&-\d \tilde \rho \Big[\begin{pmatrix}
% 0 & 1
% \\-(1+\Delta^2) & -1
%\end{pmatrix}\vec{u}{u_t} - \vec{0}{u^3}\Big] \\
%=&\left(\begin{multlined}
%\int (P_{\le N}u)^3 P_{\le N}u_t + \int{u}{u_t} \\
%+ \int\Delta^2uu_t + \int u_t(-(1+\Delta^2)u - u_t + P_{\le N}(P_{\le N}u)^3)
%\end{multlined}
% \right)\times \tilde \rho(u,u_t)\\
%=&- \Big(\int u_t^2\Big)\tilde \rho(u,u_t),
%\end{align*}
%and 
%\begin{align*}
%&\Delta_{u_t} \tilde \rho(u,u_t) \\
%&= \tr \big(\d^2_{u_t} \tilde \rho(u,u_t) \big)\\
%&= \sum_{\{h_t\} \text{ orthonormal basis of } \{P_Mu_t\} } \d^2 \tilde \rho(u,u_t) \left[\vec{0}{h_t},\vec{0}{h_t}\right] \\
%&= \sum_{\{h_t\} \text{ orthonormal basis of } \{P_Mu_t\} }  \Big(- \int h_t^2 + \left(\int u_t h_t\right)^2\Big) \tilde \rho(u,u_t),\\
%&= \Big(-(2M+1)^3 + \int u_t^2\Big) \tilde \rho(u,u_t),
%\end{align*}
%so \eqref{Fokker-Planck} is satisfied.
\end{proof}

\begin{Corollary}
The measure $\rho$ is invariant by the flow of \eqref{bvec}.
\end{Corollary}
\begin{proof}
By Corollary \ref{gwp}, one has that for every $t>0$ and every $\u_0 \in X^\alpha$, $\Phi^N_t(\u_0;\xi) \to \Phi_t(\u_0;\xi)$ in $X^\alpha$ a.s.
Let $F:X^\alpha \to \R$ be continuous and bounded. By dominated convergence and Proposition \ref{invarianceN}, we have 
\begin{align*}
&\int \E\Big[F(\Phi_t(\u_0;\xi))\Big]\d \rho(\u_0) \\
=& \int \E\Big[F(\Phi_t(\u_0;\xi))\Big]\exp\Big( -\frac 14 \int (u_0)^4 \Big) \d\mu(\u_0) \\
=& \lim_{N \to \infty} \int \E\Big[F(\Phi^N_t(\u_0;\xi)) \Big] \exp\Big( -\frac 14 \int (P_{\le N}\pi_1\u_0)^4\Big) \d\mu(\u_0)\\
=& \lim_{N \to \infty} \int F(\u_0) \exp\Big( -\frac 14 \int (P_{\le N}\pi_1\u_0)^4\Big) \d\mu(\u_0)\\
=& \int F(\u_0) \exp\Big( -\frac 14 \int (\pi_1\u_0)^4\Big) \d\mu(\u_0) \\
=& \int F(\u_0) \d \rho(\u_0).
\end{align*}

\end{proof}

\section{Ergodicity}

In this section, we proceed to show unique ergodicity for the flow $\Phi_t(\u_0;\xi)$ of \eqref{bvec}. 
We recall that, as discussed in Section 1, the flow is naturally split as $\Phi_t(\u_0;\xi) = \stick_t(\xi) + S(t) \u_0 + \v$, where
$\v = \v(\u_0,\xi;t)$ solves \eqref{veqn}.

As discussed in the introduction, the flow of \eqref{bvec} does \emph{not} satisfy the strong Feller property, so more ``standard" techniques are not applicable. Indeed, by taking a set $E_t\subset X^\alpha$ such that $\P(\{\stick_t(\xi) \in E_t\})=1$, we can see that 
$$\P(\Phi_t(\0;\xi) \in E_t + \H^2) = \P(\stick_t + \v(\0,\xi;t) \in E_t + \H^2) = \P(\stick_t \in E_t + \H^2) = 1.$$
Taking $0<\alpha< \alpha_1 < \frac12$, let $\bar\u_0 \in X^\alpha\setminus \H^{\alpha_1}$, whose existence is guaranteed by Lemma \ref{XalphaNotH2}. We have that $S(t)\bar\u_0 \not\in \H^{\alpha_1}$ for every $t$\footnote{Since $S(t)$ in invertible in $\H^{\alpha_1}$.}, and so for every $\lambda \neq 0$,
$$\P(\Phi_t(\lambda\bar\u_0;\xi) \in E_t + \H^2) = \P(\stick_t(\xi) + \lambda S(t)\bar u_0 \in E_t + \H^2).$$
By taking $E_t\subseteq \H^{\alpha_1}$, (as allowed by Proposition \ref{stickX}), we have that this probability 
is bounded from above by 
$$\P(\stick_t(\xi) + \lambda S(t)\bar u_0 \in \H^{\alpha_1}) = \P(S(t)\bar u_0 \in \H^{\alpha_1}) = 0.$$
Therefore, the function
$$\Psi(\u):= \E[\1_{\{E_t + \H^2\}}(\Phi_t(\u,\xi))]$$
satisfies $\Psi(\0) = 1$ and $\Psi(\lambda \bar\u_0) = 0$ for $\lambda \neq 0$, therefore is not continuous in $\0$. With the same 
argument, we can see that $\Psi(\H^2) = \{1\}$ and $\Psi(X^\alpha\setminus \H^{\alpha_1}) = \{0\}$, and since
both sets are dense in $X^\alpha$, we have that $\Psi$ is not continuous \emph{anywhere}.

\subsection{Restricted Strong Feller  property and irreducibility of the flow}
In this subsection, we try to recover some weaker version of the strong Feller property for the flow $\Phi$. The
end result will be to prove the following lemma, which will be crucial for the proof of ergodicity:
\begin{Lemma} \label{supports}
Let $\nu_1,\nu_2$ be two invariant measures (in the sense of \eqref{defInvariance}) such that $\nu_1\perp\nu_2$.
Then there exists some $V \subset X^\alpha$ such that $\nu_1(V)=1$ and $\nu_2(V + \H^2) = 0$.
\end{Lemma}

In order to prove this, it is convenient to introduce the space $\X = X^\alpha$ equipped with the distance
\begin{equation*}
d(\u_0,\u_1) = \norm{\u_0 - \u_1}_{\H^2} \wedge 1. 
\end{equation*}
While $\X$ is a complete metric space and a vector space, it does not satisfy many of the usual hypotheses on ambient spaces:
it is \emph{not} a topological vector space, it is \emph{disconnected}, and it is \emph{not} separable. Moreover,
the sigma-algebra $\mathscr B$ of the Borel sets on $X^\alpha$, which is also the sigma-algebra we equip $\X$
with, does \emph{not} coincide with the Borel sigma-algebra of $\X$ - $\mathscr B$ is strictly smaller\footnote{Take $\u_0 \in X^\alpha\setminus \H^2$, and let $E\subseteq \R$ be not Borel. Then it is easy to see that $E\u_0:=\{\lambda\u_0|\lambda \in E\}$ is \emph{not} in $\mathscr B$, but it is \emph{closed} in $\X$.}. 
However, in this topology, we can prove the strong Feller property.
\begin{Proposition}[Restricted strong Feller property] \label{strong_feller}
The process associated to the flow $\Phi_t(\cdot;\xi)$ of \eqref{bvec} defined on $\X$ has the strong Feller property, i.e.\ for every $t>0$, the function
\begin{equation*}
P_tG(\u) := \E[G(\Phi_t(\u,\xi))]
\end{equation*}
is continuous as a function $\X\to \R$ for every $G:\X\to \R$ measurable and bounded.
\end{Proposition}
We would like to point out that a phenomenon similar to the one described by this proposition, i.e.\ the fact that the strong Feller property holds only with respect to a stronger topology, has already been observed in the literature for other equations, for instance in \cite{fr08}.

Before being able to prove Proposition \ref{strong_feller}, we need the following (completely deterministic) lemma, 
which will take the role of support theorems for~$\xi$. 
\begin{Lemma}\label{anticonvolution}
For every $t>0$, there exists a bounded operator $T_t: \H^2 \to L^2([0,t];L^2)$ such that for every $\w \in \H^2$,
\begin{equation*}
\w = \int_0^t S(t-t') \vec{0}{\sqrt2(T_t\w)(t')} \d t' = \stick_t(T_t\w).
\end{equation*}
\end{Lemma}
\begin{proof}
This lemma is equivalent to proving that the operator $$\stick_t: L^2([0,t];L^2) \to \H^2$$ has a right inverse. Since $\H^2$ and $L^2([0,t];L^2)$ are both Hilbert spaces, we have that $\stick_t$ has a right inverse if and only if $\stick_t^\ast$ has a left inverse. 
In Hilbert spaces, this is equivalent to the estimate $\norm{\w}_{\H^2} \lesssim \norm{\stick_t^\ast \w}_{L^2([0,t];L^2)}$. We have that 
$$ (\stick_t^\ast\w)(s) = \pi_2S(t-s)^\ast \w,$$
where $\pi_2$ is the projection on the second component. Therefore, 
$$\norm{\stick_t^\ast \w}_{L^2([0,t];L^2)}^2 = \int_0^t 2\norm{\pi_2S(s)^\ast \w}_{L^2}^2.$$ 
For convenience of notation, define $L := \sqrt{\frac34 + \Delta^2}$, and define $\norm{w}_{H}^2 := \big(\int |Lw|^2\big)$ \footnote{It is easy to see that it is equivalent to the usual $H^2$ norm $\int |\sqrt{1+\Delta^2}w|^2$}. In the space $\H \cong \H^2$ given by the norm $\norm{\w}_{\H}^2 = \norm{w}_{H}^2 + \norm{w_t}_{L^2}^2$, we
have that 
$$ e^\frac t 2 \pi_2S(s)^\ast \w = L\sin(sL) v + \Big(\cos(sL)-\frac{\sin(sL)}{2L}\Big) v_t.$$
Therefore, if $\lambda_n := \sqrt{\frac34 + |n|^4}$, by Parseval
$$\norm{\stick_t^\ast \w}_{L^2([0,t];L^2)}^2 \sim_t 
\sum_{n\in\Z^3} \int_0^t \Big| \lambda_n \sin(s\lambda_n)  \widehat w(n) + \Big(\cos(s\lambda_n)-\frac{\sin(s\lambda_n)}{2\lambda_n}\Big) \widehat{w_t}(n)\Big|^2 \d s.  $$
Since by Parseval $\norm{w}_{H} = \norm{\lambda_n\widehat w}_{l^2}$ and $\norm{w_t}_{L^2} = \norm{\widehat w_t}_{l^2}$, the lemma is proven if we manage to prove that the quadratic form on $\R^2$
$$B_n(x,y) := \int_0^t \Big|\sin(s\lambda_n)x + \Big(\cos(s\lambda_n)-\frac{\sin(s\lambda_n)}{2\lambda_n}\Big) y\Big|^2\d s$$
satisfies $B_n \ge c_n \id$, with $c_n \ge \epsilon > 0$ for every $n \in \Z^3$.
We have that $B_n > 0$, since the integrand cannot be identically $0$ for $(x,y) \neq (0,0)$ 
(if the integrand is $0$, by evaluating it in $s=0$ we get $y=0$, from which evaluating in almost any other $s$ we get $x=0$). Therefore, it is enough to prove that $c_n \to c > 0$ as $|n| \to +\infty$. As $|n| \to +\infty$, 
$\lambda_n \to +\infty$ as well, so 
\begin{align*}
\lim_n &\int_0^t \sin(s\lambda_n)^2 = \frac t2,\\
\lim_n &\int_0^t \cos(s\lambda_n)^2 = \frac t2,\\
\lim_n &\phantom{\int_0^t}\frac{\sin(s\lambda_n)}{2\lambda_n} = 0,\\
\lim_n &\int_0^t \sin(s\lambda_n)\cos(s\lambda_n) = 0. 
\end{align*}
Hence, $B_n \to \frac t2 \id$ and so $c_n \to \frac t2 > 0$.
\end{proof}

\begin{proof}[Proof of Proposition \ref{strong_feller}]
We recall the decomposition $\Phi_t(\u,\xi) = S(t)\u + \stick_t(\xi) + \v(\u,\xi;t).$ For $h \in L^2_{t,x}$, adapted to the natural filtration induced by $\xi$, let 
$$\mathcal E(h):= \exp\Big(-\frac12 \int_0^t \norm{h(t')}_{L^2}^2 + \int_0^t \dual{h(t')}{\xi}_{L^2}\Big).$$
Let $C_1\gg1$, $E:= \{\xi | \norm{\stick_t(\xi)}_{C([0,t];\C^\alpha)} \le C_1\}$, and $T_t$ as in Lemma \ref{anticonvolution}. Let $\u_0\in\X$ %and $C_2\gg C_1$
. By Corollary \ref{gwp}, as long as $\xi \in E$ and 
$C_2$ is big enough (depending on $\u_0,C_1$), then 
$$\max(\norm{\v(\u,\xi;t')}_{C([0;t];\H^2)},\norm{S(t')\u + \stick_{t'}(\xi) + v}_{C([0;t];L^2)}^3) \le C_2$$
in a neighbourhood of 
$\u_0$. For convenience of notation, we denote 
\begin{equation*}
(\Tt \v(\u,\xi;t))(t') = -\frac{1}{\sqrt2}(\pi_1(S(t')\u + \stick_{t'}(\xi)) + v)^3.
\end{equation*}
Because of \eqref{veqn}, $\v$ satisfies $\v(t) = \stick_t(\Tt \v)$, and by the continuity of the flow in the initial data, $\Tt\v$ is continuous in $\u_0$. Moreover, $\Tt\v$ will always be adapted to the natural filtration induced by $\xi$. 

 By Girsanov's theorem (\cite[Theorem 1]{girsanov}), we have that 
\begin{align*}
&\phantom{=} \E[G(\Phi_t(\u,\xi))] \\
&= \E[\1_{\xi \in E^c}G(\Phi_t(\u,\xi))] + \E[\1_{\xi \in E}G(S(t)\u + \stick_t(\xi) + \v(\u,\xi;t))] \\
&=\E[\1_{\xi \in E^c}G(\Phi_t(\u,\xi))]  + \E[\1_{\xi \in E}G(S(t)\u + \stick_t(\xi + \Tt\v(\u,\xi;t)))] \\
&=\E[\1_{\xi \in E^c}G(\Phi_t(\u,\xi))] + \E[\1_{\xi \in E + \Tt\v}G(S(t)\u + \stick_t(\xi))\mathcal E(\Tt \v(\u,\xi;t))].
\end{align*}
Notice that Novikov condition ((2.1) in (\cite[Theorem 1]{girsanov})) is satisfied automatically by the estimate $\norm{\Tt\v(\u,\xi;t)}_{\H^2} \le C_2$, which
holds true on $\{\xi \in E\}$\footnote{to define a global adapted process that is equal to $\Tt \v$ on $\{\xi \in E\}$ and bounded by $C_2$ everywhere, we can for instance stop $\Tt \v$ when its norm reaches $C_2$.}.
Let $\v_0 \in \H^2$, with $\norm{\v_0}_{\H^2} \le C_2$.
\begin{align*}
&\phantom{=} \E[G(\Phi_t(\u+ \v_0,\xi))] \\
&\begin{multlined}=\E[\1_{\xi \in E^c}G(\Phi_t(\u+ \v_0,\xi))]\\
+\E[\1_{\xi \in E}G(S(t)\u + \stick_t(\xi + T_tS(t)\v_0 + \Tt\v(\u+\v_0,\xi;t)))] 
\end{multlined}\\
&\begin{multlined}=\E[\1_{\xi \in E^c}G(\Phi_t(\u+ \v_0,\xi))]\\
+\E[\1_{E+ T_tS(t)\v_0+\Tt\v}G(S(t)\u + \stick_t(\xi))\mathcal E(T_tS(t)\v_0+\Tt\v)].
\end{multlined}
\end{align*}
Up to changing $\v$ outside of $E$, we can assume $\norm{\v(\u,\xi;t)}_{\H^2} \le C_2$.
Therefore, we have (using Girsanov again)
\begin{align*}
&\phantom{=} \big|\E[G(\Phi_t(\u+ \v_0,\xi))] - \E[G(\Phi_t(\u,\xi))]\big|\\
&\begin{aligned}
\le \norm{G}_{L^\infty}&\Big(2 \P(\xi \in E^c) + \E[\1_{(E + \Tt\v)^c}\mathcal E(\Tt\v(\u,\xi;t))] \\
&+\E[\1_{(E+ T_tS(t)\v_0+\Tt\v)^c}\mathcal E(T_tS(t)\v_0+\Tt\v)] \\
&+ \E\big|\mathcal E(\Tt\v(\u,\xi;t))-\mathcal E(T_tS(t)\v_0 + \Tt\v((\u+\v_0),\xi;t) )\big|\Big)
\end{aligned}\\
&=\norm{G}_{L^\infty}(4\P(\xi \in E^c) + \E\big|\mathcal E(T_tS(t)\v_0 + \Tt\v(\u+\v_0,\xi;t))-\mathcal E(\Tt\v(\u),\xi;t)\big|)
\end{align*}
Notice that, by Burkholder-Davis-Gundy inequality, for $h$ in the form $h= T_t S(t) \w + \Tt \v$, with both $\norm{\v}_{L^2_{t,x}} \le C_2$ and $\norm{\w}_{\H^2} \le C_2$, 
\begin{align*}
\E[\exp(p\dual{h}{\xi}_{L^2_{t,x}})] &\le \sum_{k\ge0} p^k\frac1{k!}\E[|\dual{h}{\xi}_{L^2_{t,x}}|^k] \\
&\le1 + \sum_{k\ge1} p^k C^k \frac {k^{\frac k2}}{k!}\big(\norm{T_tS(t)\w}_{L^2_{t,x}}^k + \E[\norm{\Tt \v}_{L^2_{t,x}}^k] \big)\\
& \le 1 + \Psi_1(\norm{\w}_{\H^2})\norm{\w}_{\H^2} + \Psi_2(C_2)\E[\norm{\Tt \v}_{L^2_{t,x}}^2]^\frac12, \\
& \le \Psi(C_2).
\end{align*}
where $\Psi_1,\Psi_2, \Psi$ are monotone analytic functions with infinite radius of convergence. 
With the same computation, we get
$$\E[\big(\exp(p\dual{h}{\xi}_{L^2_{t,x}})-1\big)^n] \le \Psi_{3,n}(C_2) (\norm{\w}_{\H^2} + \E[\norm{\Tt \v}_{L^2_{t,x}}^2]^\frac12) \lesssim_{n,C_2}  \E[\norm{h}_{L^2_{t,x}}^2]^\frac12.$$
Therefore, by continuity of the flow of \eqref{bvec} in the initial data, for $\norm{\v_0}_{\H^2} \ll 1$, we have that
\begin{align*}
&\phantom{=}\E\big|\mathcal E(T_tS(t)\v_0 + \Tt\v(\u+\v_0,\xi;t))-\mathcal E(\Tt\v(\u),\xi;t)\big|\\
&\begin{aligned}
=&\E\Big[\exp\Big(-\frac12\norm{\Tt\v(\u,\xi;t)}_{L^2_{t,x}}^2+ \dual{\Tt\v(\u,\xi;t)}{\xi}_{L^2_{t,x}}\Big)\\
&\begin{multlined} \times \Big(\exp\Big(-\frac12(\norm{T_tS(t)\v_0 + \Tt\v(\u+\v_0,\xi;t)}_{L^2_{t,x}}^2 - \norm{\Tt\v(\u,\xi;t)}_{L^2_{t,x}}^2)\\
+ \dual{T_tS(t)\v_0 + \Tt\v(\u+\v_0,\xi;t)-\Tt\v(\u,\xi;t)}{\xi}_{L^2_{t,x}}\Big) - 1 \Big)\Big]
\end{multlined}\\
=&\E\Big[\exp\Big(-\frac12\norm{\Tt\v(\u,\xi;t)}_{L^2_{t,x}}^2+ \dual{\Tt\v(\u,\xi;t)}{\xi}_{L^2_{t,x}}\Big)\\
&\resizebox{.97\hsize}{!}{$\displaystyle \begin{multlined}\times\Big(
\exp\Big(-\frac12(\norm{T_tS(t)\v_0 + \Tt\v(\u+\v_0,\xi;t)}_{L^2_{t,x}}^2 - \norm{\Tt\v(\u,\xi;t)}_{L^2_{t,x}}^2)\Big)-1\Big)\\
\times\exp\Big( \dual{T_tS(t)\v_0 + \Tt\v(\u+\v_0,\xi;t)-\Tt\v(\u,\xi;t)}{\xi}_{L^2_{t,x}}\Big)
\end{multlined}$}\\
&\phantom{\times\Big(}\,+\exp\Big( \dual{T_tS(t)\v_0 + \Tt\v(\u+\v_0,\xi;t)-\Tt\v(\u,\xi;t)}{\xi}_{L^2_{t,x}}\Big)  - 1 \Big)\Big]
\end{aligned}\\
&\le \Big[\E\exp\Big(2\dual{\Tt\v(\u,\xi;t)}{\xi}_{L^2_{t,x}}\Big)\Big]^\frac12\\
&\phantom{\le} \resizebox{.97\hsize}{!}{$\displaystyle\times\Big[\Big(\E\Big(\exp\Big(-\frac12(\norm{T_tS(t)\v_0 + \Tt\v(\u+\v_0,\xi;t)}_{L^2_{t,x}}^2 - \norm{\Tt\v(\u,\xi;t)}_{L^2_{t,x}}^2)\Big)-1\Big)^4\Big)^\frac14$}\\
&\phantom{\le\times\Big[}\times\Big(\E\exp\Big(4\dual{T_tS(t)\v_0 + \Tt\v(\u+\v_0,\xi;t)-\Tt\v(\u,\xi;t)}{\xi}_{L^2_{t,x}}\Big)\Big)^\frac14
\\
&\phantom{\le \times\Big[} + \resizebox{.89\hsize}{!}{$\displaystyle\Big(\E\Big(\exp\Big( \dual{T_tS(t)\v_0 +\Tt\v(\u+\v_0,\xi;t)-\Tt\v(\u,\xi;t)}{\xi}_{L^2_{t,x}}\Big)  - 1 \Big)^2\Big)^\frac12\Big]$}\\
&\lesssim_{C_2}% \resizebox{.95\hsize}{!}{$
%\displaystyle\Psi(CC_2)
[\E\norm{T_tS(t)\v_0 + \Tt\v(\u+\v_0,\xi;t)-\Tt\v(\u,\xi;t)}_{L^2_{t,x}}^2]^\frac12
%(\Psi(CC_2) + 1)
%$}
,
\end{align*}
which is converging to $0$ as $\norm{\v_0}_{\H^2} \to 0$ because of dominated convergence. Therefore,
\begin{align*}
&\limsup_{\norm{v_0}_{\H^2}\to 0} \big|\E[G(\Phi_t(\u+ \v_0,\xi))] - \E[G(\Phi_t(\u,\xi))]\big| \\
\le& \limsup_{\norm{v_0}_{\H^2}\to 0}\norm{G}_{L^\infty}\left(
\begin{multlined}
4\P(\xi \in E^c) \\
+ \E\big|\mathcal E(T_tS(t)\v_0 + \Tt\v(\u+\v_0,\xi;t))-\mathcal E(\Tt\v(\u),\xi;t)\big|
\end{multlined}\right)\\
=& 4\norm{G}_{L^\infty}\P(\xi \in E^c).
\end{align*}
Since the left-hand-side does not depend on $C_1$, we can send $C_1 \to \infty$, and we obtain that 
$$\lim_{\norm{v_0}_{\H^2}\to 0} \big|\E[G(\Phi_t(\u+ \v_0,\xi))] - \E[G(\Phi_t(\u,\xi))]\big| = 0,
 $$
i.e.\ $\E[G(\Phi_t(\u,\xi))]$ is continuous in $\u$ in the $\X$ topology.
\end{proof}

While the topology of $\X$ does not allow to extend many common consequences of the strong Feller property, 
we still have the following generalisation of the disjoint supports property.
\begin{Corollary}\label{feller_supports}
Let $\nu_1\perp\nu_2$ be two invariant measures. Then there exists a measurable
open set $V_0 \subseteq \X$ such that $\nu_1(V_0) = 1$ and $\nu_2(V_0)=0$.
\end{Corollary}
\begin{proof}
Let $S_1 \subset\X$ be a measurable set with $\nu_1(S_1) = 1$, $\nu_2(S_1) = 0$. Recall that a set is measurable if and only if it is Borel in $X^\alpha$. Consider the function
$$\Psi(\u):= \E[\1_{S_1}(\Phi_t(\u,\xi))]. $$
By the Proposition \ref{strong_feller}, $\Psi:\X\to \R$ is continuous. Moreover, since $S_1$ is a Borel set in $X^\alpha$, $\Psi$ is also measurable. 
By invariance of $\nu_j$, $\Psi = 1$ $\nu_1$-a.s. and $\Psi = 0$ $\nu_2$-a.s. Let $V_0:=\{\Psi > \frac12\}$. We have that $V_0\subset\X$ is open 
by continuity of $\Psi$, it is measurable since $\Psi$ is measurable, 
$$\nu_1(V_0) \ge \nu_1(\{\Psi=1\}) = 1 $$
and 
$$\nu_2(V_0) \le \nu_2(\{\Psi\neq0\}) = 0. $$
\end{proof}
\begin{Lemma}[Irreducibility] \label{irreducibility}
Suppose that $\nu$ is invariant for the flow of \eqref{bvec}, and let $E\subset X^\alpha$ such that $\nu(E) = 0$. 
Then for every $\w \in \H^2$, $\nu(E+\w) = 0$.
\end{Lemma}
\begin{proof}
Since $X^\alpha$ is a Polish space, by inner regularity of $\nu$ it is enough to prove the statement when $E$ is
compact. Take $C_1 < +\infty$, and let 
$$F:=\{\xi:\norm{\stick_{\cdot}(\xi)}_{C([0,t];\C^\alpha)} \le C_1\}.$$
Proceeding in a similar way to Proposition \ref{strong_feller}, we have that by the compactness of $E$, the boundedness of $\stick_t(\xi)$ and Proposition \ref{gwp}, $\Tt \v$ satisfies Novikov's condition on $\{\xi \in F\}$ and
\begin{align*}
0 = \nu(E) &= \int \E[\1_{E}(S(t)\u + \stick_t(\xi) + \v])\d\nu(\u)\\
&\ge \int \E[\1_{F}(\xi) \1_{E}(S(t)\u + \stick_t(\xi) + \v)]\d\nu(\u) \\
&= \int \E[\1_{F+\Tt\v}(\xi)\1_{E}(S(t)\u + \stick_t(\xi))\mathcal E(\Tt\v)] \d\nu(\u).
\end{align*}
Since $\mathcal E > 0$ $\P\times\nu-$a.s., this implies that $\1_{F+\Tt \v}(\xi)\1_{E}(S(t)\u + \stick_t(\xi))  = 0$ $\P\times\nu-$a.s.
By sending $C_1\to \infty$, by monotone convergence we obtain that 
$\1_{E}(S(t)\u + \stick_t(\xi)) = 0$ $\P\times\nu-$a.s.

Let $\w \in \H^2$. Then, proceeding similarly, 
\begin{align*}
&\int \E[\1_{F}(\xi) \1_{E+\w}(S(t)\u + \stick_t(\xi) + \v)]\d\nu(\u)\\
=&\int \E[\1_{F}(\xi) \1_{E}(S(t)\u + \stick_t(\xi) + \v-\w)]\d\nu(\u)\\
=&\int \E[\1_{F+\Tt\v-T_t\w}(\xi) \1_{E}(S(t)\u + \stick_t(\xi))\mathcal E(\Tt\v-T_t\w)] \d \nu(\u) = 0,
\end{align*}
since the integrand is $0$ $\P\times\nu-$a.s.
By taking $C_1 \to \infty$, by monotone convergence we get 
\begin{align*}
0 =&\int \E[\1_{E+\w}(S(t)\u + \stick_t(\xi) + \v)]\d\nu(\u)\\
=&\nu(E+\w).
\end{align*}
\end{proof}

\begin{proof}[Proof of Lemma \ref{supports}]
Let $\nu_1\perp \nu_2$ be two invariant measures, let $V=V_0$ be the set given by Corollary \ref{feller_supports}, and let $\{\w_n\}_{n\in\N}$ be a countable dense subset of $\H^2$. We have that, by definition,
$\nu_1(V) = 1$ and $\nu_2(V) = 0$. By Lemma \ref{irreducibility}, $\nu_2(V + \w_n) = 0$ for 
every $\w_n$. Therefore, $\nu_2(\bigcup_n (V+ \w_n)) = 0$. Moreover, since $V$ is open in $\X$, 
we have that $\bigcup_n (V+ \w_n) = V + \H^2$. Therefore, $\nu_2(V+\H^2)=0$. 
\end{proof}
\subsection{Projected flow}
In this subsection, we will bootstrap ergodicity of the measure $\rho$ from ergodicity of the flow of the
\emph{linear} equation
\begin{equation} \label{linear}
\partial_t \vec{u}{u_t} =
-\begin{pmatrix}0 & -1 \\ 1+\Delta^2 & 1\end{pmatrix} \vec{u}{u_t} + \vec{0}{\sqrt2 \xi}.
\end{equation}
 The measure $\mu$ defined in \eqref{mu} is invariant for the flow of this equation (which can be seen
 as a special case of Proposition \ref{invarianceN} for $N=-1$). Let $L(t)\u$ be the flow of \eqref{linear}, i.e.\ 
 \begin{equation*}
 L(t)\u := S(t)\u + \stick_t(\xi).
 \end{equation*}
 \begin{Lemma}\label{muergodic}
 The measure $\mu$ is the only invariant measure for \eqref{linear}. Moreover, for every $\u_0 \in X^\alpha$, the law of $L(t)\u_0$ is 
 weakly converging to $\mu$ as $t \to \infty$.
 \end{Lemma}
\begin{proof}
Let $\u_0, \u_1 \in X^\alpha$, and let $F: X^\alpha \to \R$ be a Lipschitz function. We have that 
\begin{align*}
\big|\E[F(L(t)\u_0) - F(L(t)\u_1)]\big| & = \big|\E[F(S(t)\u _0+ \stick_t(\xi)) - F(S(t)\u_1 + \stick_t(\xi))]\big| \\
& \le \E[\min(\Lip(F)\norm{S(t)\u_0 - S(t)\u_1}_{X^\alpha}, \norm{F}_{L^\infty})] \\
&\le \min (e^{-\frac t 8} \Lip(F)\norm{\u_0 - \u_1}_{X^\alpha}, \norm{F}_{L^\infty})
\end{align*}
Therefore by invariance of $\mu$, we have that 
\begin{align*}
&\Big|\E[F(L(t)\u_0)] - \int F(\u_1) \d \mu(\u_1)\Big| \\
=& \Big|\int \big(\E[F(L(t)\u_0) -\E[F(L(t)\u_1)]\big) \d \mu(\u_1)\Big|\\
\le& \int \min (e^{-\frac t 8} \Lip(F)\norm{\u_0 - \u_1}_{X^\alpha}, \norm{F}_{L^\infty}) \d \mu(\u_1),
\end{align*}
which is converging to $0$ by dominated convergence. Since Lipschitz functions are dense in the set of continuous functions, this implies 
that the law of $L(t)\u_0$ is weakly converging to $\mu$. Similarly, if $\nu$ is another invariant measure, 
\begin{align*}
&\phantom{=}\Big| \int F(\u_0) \d \nu(\u_0) - \int F(\u_1) \d \mu(\u_1)\Big| \\
&= \Big|\iint \big(\E[F(L(t)\u_0) -\E[F(L(t)\u_1)]\big)\d \nu(\u_0) \d \mu(\u_1)\Big|\\
&\le \iint \min (e^{-\frac t 8} \Lip(F)\norm{\u_0 - \u_1}_{X^\alpha}, \norm{F}_{L^\infty}) \d \nu(\u_0) \d \mu(\u_1),
\end{align*}
which is converging to $0$ by dominated convergence. Since the left hand side does not depend on $t$, one gets that $\int F(\u_0) \d \nu(\u_0) = \int F(\u_1) \d \mu(\u_1)$ for every $F$ Lipschitz, so $\mu = \nu$. 
\end{proof}

Consider the (algebraic) projection $\pi: X^\alpha \to X^\alpha/\H^2$. While the quotient space does 
not have a sensible topology, we can define the quotient sigma-algebra, 
\begin{equation*}
\A:= \{ F \subseteq X^\alpha/\H^2 \text{ s.t. } \pi^{-1}(F)\subseteq X^\alpha \text{ Borel}\}, 
\end{equation*}
which corresponds to the finest $\sigma$-algebra that makes the map $\pi$ measurable. While this will not be relevant in the following, we can see that $\A$ is 
relatively rich: if $E\subset X^\alpha$ is closed and $B$ is the closed unit ball in $\H^2$, since $B$ is compact in $X^\alpha$,  
$E+nB$ is closed for every $n$, so $E+ \H^2 = \bigcup_n  E + nB$ is Borel. Therefore, $\pi(E) \in \A$.

Since $S(t)$ maps $\H^2$ into itself, is it easy to see that if $\pi(\u)=\pi(\v)$, then $\pi(L(t)\u) = \pi(L(t)\v)$.
Therefore, $\pi (L(t) \u)$ is a function of $\pi(\u)$, and we define 
\begin{equation*}
\bar L(t)\pi(\u):= \pi(L(t)\u).
\end{equation*}
 Moreover, if $\Phi_t(\u;\xi) = S(t)\u + \stick_t(\xi) + \v(\u,\xi;t)$ is the flow of \eqref{bvec}, where $\v$ solves \eqref{veqn}, since $\v$ belongs to $\H^2$,
 we have that 
\begin{align*}
\pi(\Phi_t(\u;\xi)) \!=\! \pi(S(t)\u + \stick_t + \v(\u,\xi;t)) \!=\! \pi(S(t)\u + \stick_t) \!=\! \pi(L(t)\u) \!=\! \bar L(t)\pi(\u).
\end{align*}
 Therefore, also $\pi(\Phi_t(\u;\xi))$ is a function of $\pi(\u)$, and moreover 
\begin{equation} \label{piPhi=piL}
\pi(\Phi_t(\u;\xi)) = \bar L(t)\pi(\u),
\end{equation}
so the projections of the flows for \eqref{linear} and \eqref{bvec} coincide. 
\begin{Proposition} \label{pisharpergodicity}
The measure $\pi_\sharp(\mu)$ is ergodic for the process associated to $\bar L(t): X^\alpha/\H^2 \to X^\alpha/\H^2$. 
\end{Proposition}
\begin{proof}
If $G:X^\alpha/\H^2 \to \R$ is a bounded measurable function, then by invariance of $\mu$,
\begin{equation} \label{piinvariance}
\begin{aligned}
\int \E[G(\bar L(t)x)] \d\pi_\sharp\mu(x) &= \int \E[G(\bar L(t)\pi(\u))] \d \mu(\u) \\
& = \int \E[G(\pi(L(t)\u))] \d\mu(\u) \\
& = \int G(\pi(\u)) \d\mu(\u) \\
& = \int G(x) \d\pi_\sharp\mu(x),
\end{aligned}
\end{equation}
so $\pi_\sharp\mu$ is invariant.

Let now $G$ be a function such that $\E[G(\bar L(t) x)] = G(x)$ for $\pi_\sharp\mu$-a.e.\ $x~{\!\in\!}~X^\alpha/\H^2$\!.
Then 
$$\E[G\circ\pi(L(t)\u)]\E[G(\pi(L(t) \u))] = E[G(\bar L(t) \pi(\u))] = G(\pi(\u)),$$
so $G\circ \pi$ is $\mu$-a.s.\ constant by ergodicity of $\mu$. Therefore, $G$ is $\pi_\sharp\mu$-a.s.\ constant,
so $\pi_\sharp\mu$ is ergodic.
%Let $F$ be an invariant set for $\bar L(t)$. Therefore, $E:=\pi^{-1}(F)$ will be invariant for $L(t)$. 
%By ergodicity of $\mu$, this implies that $\pi(E) = 0$ or $\pi(E)=1$. Therefore, $\pi_\sharp(F) = \pi(E) \in \{0,1\}$,
%so $\bar L(t)$ is ergodic.
\end{proof}
\begin{Remark} \label{0-1}
We can see $\mu$ as the law of the random variable $\u$ defined in \eqref{u(omega)}. 
In this way, for every $E\subset X^\alpha$, by definition $\mu(E) = \P(\{\u \in E\})$. If $E = E + \H^2$, then 
the event $\{\u \in E\}$ is independent from $\{g_n,h_n| |n|<N\}$ for every $N$, since $\H^2$ contains every 
function with finite Fourier support. Therefore, $E \in \bigcap_N \sigma(g_n,h_n| |n|\ge N)$.
By Kolmogorov's 0-1 theorem, this implies that $\mu(E) = 0$ or $\mu(E) = 1$.

Since by definition 
$\pi_{\sharp}\mu(F) = \mu(\pi^{-1} (F))$ and $\pi^{-1}(F) =  \pi^{-1}(F) + \H^2$, then
for \emph{every} set $F\in \A$ we have $\pi_{\sharp}\mu(F) \in \{0,1\}$, therefore trivially any invariant set has
measure $0$ or $1$, hence the measure $\pi_\sharp\mu$ is ergodic.
\end{Remark}
\begin{Proposition}\label{ergodicity}
Let $\nu$ be an invariant measure for the flow of \eqref{bvec} such that $\pi_\sharp \nu \ll \pi_\sharp \mu$. Then $\nu = \rho$.\footnote{Notice that since $\rho \ll \mu$, then $\pi_\sharp \rho \ll \pi_\sharp \mu$.}
\end{Proposition}
\begin{proof}
Suppose by contradiction that $\nu \neq \rho$. Let 
\begin{align*}
\rho_1 = \frac{1}{(\rho - \nu)_{+}(X^\alpha)} (\rho - \nu)_{+}\\
\rho_2 = \frac{1}{(\nu-\rho)_{+}(X^\alpha)}(\nu - \rho)_{+}.
\end{align*}
Since $\rho,\nu$ are invariant, it is easy to see that $\rho_1, \rho_2$ are both invariant probabilities.
%\footnote{Just use the characterisation $$\text{If }F\ge0, \int F(x) \d \sigma_+(x) = \sup_{0 \le G \le F} \int G(x) \d \sigma(x). $$}.
Moreover, $\rho_1 \perp \rho_2$, and $\rho_j \ll \rho + \nu$, so $\pi_\sharp \rho_j \ll \pi_\sharp \rho + \pi_\sharp \nu \ll \pi_\sharp \mu$.

Proceeding as for \eqref{piinvariance}, and using \eqref{piPhi=piL}, we have 
\begin{align*}
\int \E[G(\bar L(t)x)] \d\pi_\sharp\rho_j(x) &= \int \E[G(\bar L(t)\pi(\u))] \d \rho_j(\u) \\
& = \int \E[G(\pi(\Phi_t(\u;\xi)))] \d\rho_j(\u) \\
& = \int G(\pi(\u)) \d\rho_j(\u) \\
& = \int G(x) \d\pi_\sharp\rho_j(x),
\end{align*}
therefore $\pi_\sharp\rho_j$ is invariant. Moreover, since $\pi_\sharp\rho_j\ll \pi_\sharp\mu$, by invariance of $\pi_\sharp \rho_j$ and ergodicity of $\pi_\sharp \mu$,
we must have $\pi_\sharp\rho_j = \pi_\sharp\mu$. 
Let now $V$ be the set given by Lemma \ref{supports}, i.e.\ $\rho_1(V) = 1$, $\rho_2(V + \H^2) = 0$.
We have 
\begin{align*}
 0 &= \rho_2(V+\H^2) = \pi_\sharp\rho_2(\pi(V + \H^2))=\pi_\sharp\mu(\pi(V + \H^2))\\
&=\pi_\sharp\rho_1(\pi(V + \H^2))=\rho_1(V+\H^2) \ge \rho_1(V) = 1,
\end{align*}
contradiction.
\end{proof}

\begin{Remark}
Using Remark \ref{0-1}, it is possible to show $\pi_\sharp\rho_j=\pi_\sharp\mu$ without using the ergodicity of 
$\pi_\sharp\mu$. We have indeed that $\rho_j\ll\mu$ implies $\pi_\sharp\rho_j\ll\pi_\sharp\mu$. Let 
$E$ be any set with $\pi_\sharp\rho_j(E)>0$. Then by absolute continuity, 
$\pi_\sharp\mu(E) > 0$ as well, and by Remark \ref{0-1}, $\pi_\sharp\mu(E) = 1 \ge \pi_\sharp\rho_j(E)$.  
Therefore $\pi_\sharp\rho_j\le\pi_\sharp\mu$, and since they are both probabilities, we must have 
$\pi_\sharp\rho_j=\pi_\sharp\mu$.
\end{Remark}
\begin{Corollary} \label{ergodicity2}
The measure $\rho$ is ergodic for the Markov process associated to the flow $\Phi_t(\cdot,\xi):X^\alpha\to X^\alpha$ of \eqref{bvec}. 
\end{Corollary}
\begin{proof}
Let $\nu \ll \rho$, $\nu$ invariant. We have that $\pi_\sharp \nu \ll \pi_\sharp \rho \ll \pi_\sharp \mu$. Hence, by Proposition \ref{ergodicity}, $\nu = \rho$. Therefore, $\rho$ is ergodic.
\end{proof}

We conclude this section by proving unique ergodicity for the measure $\rho$. This will be the only part of this
paper for which we require the good long-time estimates for the flow given by \eqref{vgrowth} (up to this
point, whenever we used Corollary \ref{gwp}, we needed just the qualitative result of global existence and 
\emph{time-dependent} bounds on the growth of the solution).

In particular, we will prove the following version of Birkhoff's theorem for this process, which in particular implies
Theorem \ref{mainthm}.
\begin{Proposition}
Let $\Phi_t(\u;\xi)$ be the flow of \eqref{bvec}. For every $u_0 \in X^\alpha$, we have that $\rho_t \rightharpoonup \rho$ as $t \to \infty$, where $\rho_t$ is defined by
$$ \int F(\u)\d \rho_t(\u) := \frac1t \int_0^t\E[F(\Phi_{t'}(\u_0,\xi))] \d t'.$$
\end{Proposition}
\begin{proof}
Consider the usual decomposition 
$$\Phi_t (\u_0;\xi) = S(t)\u_0 + \stick_t(\xi) + \v(\u_0,\xi;t).$$
We have that the law $\mu_t$ of $S(t)\u_0+ \stick_t(\xi) = L(\u_0)$ is tight in $X^\alpha$, because by Lemma \ref{muergodic}, $\mu_t \rightharpoonup \mu$ as $t \to \infty$. 
Therefore, there exists compact sets $K_\epsilon \subseteq X^\alpha$ such that $\P(\{S(t)\u_0+ \stick_t(\xi) \in K_\epsilon\}) \ge 1 - \epsilon$. 
Moreover, by the estimate \eqref{vgrowth} and the compactness of the embedding 
$\H^2 \hookrightarrow X^\alpha$, we have that also the law of $\v$ is tight; more precisely, there exists constants $c_\epsilon$ such that 
$\P(\{\norm{\v}_{\H^2} \le c_\epsilon\})\ge 1- \epsilon$, uniformly in $t$.
Therefore,
$$\P(\big\{\Phi_t(\u_0,\xi) \in K_\epsilon + \{\norm{\cdot}_{\H^2} \le c_\epsilon\}\big\}) \ge 1-2\epsilon, $$
so also the law of $\Phi_t(\u_0,\xi)$ is tight. By averaging in time, we obtain that also the sequence $\rho_t$ is. Hence it is enough to prove that 
every weak limit point $\bar \rho$ of $\rho_t$ satisfies $\bar \rho = \rho$. Notice that, by definition, $\bar\rho$ is invariant.
Let $t_n \to \infty$ be 
a sequence such that $\rho_{t_n} \rightharpoonup \bar\rho$. Consider the random variable 
\begin{equation*}
Y_t := (S(t)\u_0 + \stick_t(\xi), \v(\u_0,\xi;t)) \in X^\alpha \times X^\alpha.
\end{equation*}
By the same argument, the law $Y_t$ is tight in $X^\alpha \times X^\alpha$, with 
compact sets $C_\epsilon$ of the form $C_\epsilon = K_\epsilon\times \{\norm{y}_{\H^2} \le c_\epsilon\}$ such that $\P(Y_t \in C_\epsilon) \ge 1 - \epsilon$ uniformly in $t$. Therefore, tightness with the same associated compact sets will hold for the measure $\nu_t$ given by 
\begin{equation*}
\int F(\u_1,\u_2) \d \nu_t (\u_1,\u_2) = \frac1t\int_0^t \E[F(Y_t)] \d t.
\end{equation*}
Hence, up to subsequences, $\nu_{t_n} \rightharpoonup \nu$, with $\nu$ concentrated on $X^\alpha \times \H^2$. 
Define the maps 
$\mathfrak S,\pi_1: X^\alpha \times X^\alpha \to X^\alpha$ by
\begin{align*}
\mathfrak S(x,y)&:= x+y,\\
\pi_1(x,y) &:= x.
%\pi_2(x,y) &:= y.
\end{align*}
Since $\mathfrak S(Y_t) =  \Phi_t(\u_0,\xi)$, then $\mathfrak S_\sharp\nu=\bar\rho$. Moreover, since $\pi_1(Y_t) = S(t)\u_0 + \stick_t(\xi)$, we have that $(\pi_1)_\sharp\nu_t = \mu_t$, so $(\pi_1)_\sharp \nu = \mu$. Recall 
the projection $\pi:X^\alpha \to X^\alpha/\H^2$. On $X^\alpha \times \H^2$, we have that
$\pi \circ \mathfrak S = \pi \circ \pi_1$. Therefore, since $\nu$ is concentrated on $X^\alpha \times \H^2$, 
$$\pi_\sharp\bar\rho = \pi_\sharp\mathfrak S_\sharp \nu = \pi_\sharp(\pi_1)_\sharp\nu = \pi_\sharp \mu.$$
Hence, by Proposition \eqref{ergodicity}, we get $\bar \rho = \rho$.
%Suppose now by contradiction that $\bar \rho \neq \rho$. By definition, $\bar \rho$ is invariant, so by ergodicity of $\rho$, this implies that 
%the singular part $\sigma$ of $\bar \rho$ with respect to $\rho$ is not trivial. Up to normalising, we obtain that there 
%exists an invariant measure $\sigma \ll \bar \rho$ such that $\sigma \perp \rho$. Since 
%$\sigma \ll \bar\rho$, then $\pi_\sharp \sigma \ll \pi_\sharp \bar \rho = \pi_\sharp\mu$. Therefore, we have 
%two probabilities $\sigma\perp\rho$ such that $  \pi_\sharp \sigma,\pi_\sharp \rho \ll \pi_\sharp\mu$.
%
%We get the contradiction arguing exactly as in the proof of Proposition \ref{ergodicity}. Indeed, let $V$ be 
%the set given by Lemma \ref{supports}, i.e.\ $\sigma(V + \H^2)=0$, $\rho(V)~=~1$. Repeating the computation 
%of \eqref{piinvariance}, we have that $\pi_\sharp \sigma, \pi_\sharp \rho$ are invariant for $\bar L(t)$, hence by ergodicity 
%of $\pi_\sharp\mu$, they are both equal to $\pi_\sharp\mu$. Therefore
%$$0 = \sigma(V + \H^2) = \pi_\sharp \sigma(\pi(V)) = \pi_\sharp \mu(\pi(V)) = \pi_\sharp \rho(\pi(V)) \ge \rho(V)=1,$$
%contradiction.
\end{proof}
\begin{Remark}
If we could improve Proposition \ref{pisharpergodicity}
to \emph{unique} ergodicity for the measure 
$\pi_\sharp\mu$, we would automatically improve the result of 
Corollary \ref{ergodicity2} to \emph{unique}
ergodicity, without using at all the long time estimates for the growth of $\v$. Indeed, in the proof of Proposition
\ref{ergodicity}, the only point in which we used the condition $\rho_j \ll \rho$ was for showing that 
$\pi_\sharp \rho_j = \pi_\sharp\mu$. If we knew that the measure $\pi_\sharp\mu$ was uniquely ergodic, then 
$\pi_\sharp \rho_j = \pi_\sharp\mu$ will follow automatically from invariance, without the need for the extra condition $\rho_j \ll \rho$.
\end{Remark}
\section*{Acknowledgements}
The author would like to thank his PhD supervisor Tadahiro Oh for suggesting this problem and his continuous  help and support in the preparation of this work. He would also like to thank J. Forlano and P. Sosoe for reading the draft of this paper, and their many useful suggestions and corrections, and the unnamed referees, for their many suggestions on how to improve the paper and make it more readable.

The author was supported by the European Research Council (grant no. 637995 ``ProbDynDispEq") and by The Maxwell Institute Graduate School in Analysis and its Applications, a Centre for Doctoral Training funded by the UK Engineering and Physical Sciences Research Council (grant EP/L016508/01), the Scottish Funding Council, Heriot-Watt University and the University of Edinburgh.

\begin{minipage}{\linewidth}
\textsc{\hfill\\
Leonardo Tolomeo\\
Mathematical Institute\\ Hausdorff Center for Mathematics \\
Universit\"at Bonn\\
}
\texttt{tolomeo@math.uni-bonn.de}
\end{minipage}

\end{document}